\documentclass[12pt]{amsart}
\usepackage{amsmath,amssymb,amsthm,graphicx, url, verbatim, float,hyperref,ulem}
\usepackage{enumerate}
\usepackage{xcolor, relsize}

\newtheorem{lemma}{Lemma}[section]
\newtheorem{proposition}[lemma]{Proposition}
\newtheorem{theorem}[lemma]{Theorem}
\newtheorem{corollary}[lemma]{Corollary}
\newtheorem{question}[lemma]{Question}

\theoremstyle{definition}
\newtheorem{definition}[lemma]{Definition}
\newtheorem{remark}[lemma]{Remark}

\newtheorem*{namedtheorem}{\theoremname}
\newcommand{\theoremname}{testing}

\newcommand{\Z}{\mathbb{Z}}

\newcommand{\Q}{\mathbb{Q}}
\newcommand{\C}{\mathbb{C}}

\newcommand{\F}{\mathbb{F}}

\newcommand{\Sp}{\mathrm{Sp}}
\newcommand{\id}{\mathrm{id}}
\DeclareMathOperator{\Ab}{\mathrm{Ab}}
\DeclareMathOperator{\Mod}{\mathrm{Mod}}
\DeclareMathOperator{\Ker}{\mathrm{Ker}}
\DeclareMathOperator{\im}{\mathrm{Im}}

\title[On the kernel of $\mathrm{SO}(3)$-WRT quantum representations]{On the kernel of $\mathrm{SO}(3)$-Witten-Reshetikhin-Turaev quantum representations}
\author{Renaud Detcherry}
\date{} 

\address{Institut de Math\'ematiques de Bourgogne, UMR 5584 CNRS, Universit\'e Bourgogne Franche-Comt\'e, F-2100 Dijon, France}
\email{renaud.detcherry@u-bourgogne.fr}

\author{Ramanujan Santharoubane}
\address{Laboratoire de math\'ematique d'Orsay, UMR 8628 CNRS,
B\^atiment 307, Universit\'e Paris-Saclay, 
91405 ORSAY Cedex, FRANCE}
\email{ramanujan.santharoubane@universite-paris-saclay.fr}

\begin{document}

\begin{abstract}In this paper, we study the kernels of the $\mathrm{SO}(3)$-Witten-Reshetikhin-Turaev  quantum representations $\rho_p$ of mapping class groups of closed orientable surfaces $\Sigma_g$ of genus $g.$ We investigate the question whether the kernel of $\rho_p$ for $p$ prime is exactly the subgroup generated by $p$-th powers of Dehn twists. We show that if $g\geq 3$ and $p\geq 5$ then $\Ker \rho_p$ is contained in the subgroup generated by $p$-th powers of Dehn twists and separating twists, and if $g\geq 6$ and $p$ is a large enough prime then $\Ker \rho_p$ is contained in the subgroup generated by the commutator subgroup of the Johnson subgroup and by $p$-th powers of Dehn twists.  
\end{abstract}
\maketitle

\section{Introduction}
\label{sec:intro}

For $\Sigma$ a compact connected oriented surface, let $\Mod(\Sigma)$ be its mapping class group.  Among the many constructions of finite dimensional linear representations of $\mathrm{Mod}(\Sigma),$ the Witten-Reshetikhin-Turaev quantum representations, introduced by Witten \cite{W89} and rigorously defined by Reshetikhin and Turaev \cite{RT91} have many striking properties. The Witten-Reshetikhin-Turaev theory associates a family of representations to any compact Lie group, however in this article we will focus on the so-called $\mathrm{SO}(3)$-WRT representations. For any odd integer $p\geq 3$, let $K_p$ be the cyclotomic field $\Q[e^{2 i \pi /p}]$. For any compact connected oriented surface $\Sigma$, the $\mathrm{SO}(3)$-quantum representation is a projective representation
$$\rho_p: \Mod(\Sigma) \longrightarrow \mathrm{PGL}_{d}( K_p),$$
If $\Sigma$ has boundary, the representation depends on the specification of integers between $0$ and $p-2$ on each boundary compenent, this data is refered to as boundary colors. The integer $d$ is given by the celebrated Verlinde formula which is an explicit formula. As the coefficients of the representation lie in a number field, the representation can be modified using different embeddings of $K_p$ in $\C$, this is equivalent to say that the representation depends on a choice of a primitive $p$-th root of unity $\zeta_p$. Moreover, for an appropriate choice of $\zeta_p\in \C^*,$ those are unitary representations.

 The $\mathrm{SO}(3)$-WRT quantum representations give examples of finite dimensional (projective) unitary representations with infinite image\cite{F99}\cite{M99} of $\Mod(\Sigma),$ a phenomenon not observed for any other known representations of mapping class groups of surfaces of genus at least $2.$ The $\mathrm{SO}(3)$-quantum representations seem to capture a lot of mysterious and deep information about the mapping class groups. First, they are asymptotically faithfully by \cite{FWW02} or \cite{A06}, which can be used to recover that the mapping class groups of surfaces are residually finite. Moreover, the $\mathrm{SO}(3)$-quantum representations have been used to show that every finite group is involved in the mapping class group $\Mod(\Sigma_g)$ of any closed surface of genus $g\geq 2$ in \cite{MR12}, or to construct finite covers of surfaces whose homology is not spanned by lifts of simple closed curves \cite{KS16}.
 
A lot remains unknown about the kernel and images of the representations $\rho_p.$ While they are asymptotically faithful, at fixed $p,$ the representation $\rho_p$ is never faithful. Indeed if $t_{\alpha}$ is the Dehn twist along any simple closed curve $\alpha$ on $\Sigma,$ then $t_{\alpha}^p \in \Ker \rho_p.$ However, at the time of the writing, the only known kernel elements are products of $p$-th powers of Dehn twists, when the surface has genus $g\geq 3.$ Moreover, recent work of Deroin and March\'e\cite{DM22} show that the subgroup of $\Mod(\Sigma)$ generated by $p$-th powers of Dehn twists has finite index in the kernel of $\rho_p$, when $p=5$ and $\Sigma$ is the surface $\Sigma_{g,n}$ of genus $g$ with $n$ boundary components, with $(g,n)\in \lbrace (0,4),(0,5),(1,2),(1,3),(2,1) \rbrace.$ This raises the following question:
\begin{question}\label{conj:noyau} Let $p\geq 5$ be an odd prime, let $\Sigma$ be a closed compact connected orientable surface of genus $g\geq 3,$ let $\rho_p$ be the $\mathrm{SO}(3)$-WRT quantum representation of $\Mod(\Sigma).$ Is it true that $\Ker \rho_p=T_p,$ the subgroup generated by $p$-th powers of Dehn twists ?
\end{question} This question was asked explicitly by Masbaum in  \cite{M08}. 
We note that the restriction to $p$ prime is motivated by the fact that $\rho_p$ is irreducible when $p$ is prime \cite{R01}, while for $p$ composite it is not necessarily irreducible \cite{Kor19}.

Our approach to studying Question \ref{conj:noyau} is to use the $h$-adic expansion of quantum representations, introduced by Gilmer and Masbaum \cite{GM07} and studied in \cite{GM14}\cite{GM17} and \cite{KS16}. First, we note that by a theorem of Gilmer and Masbaum, for $p$ prime the representations $\rho_p$ are (up to conjugation) valued in $\mathrm{PGL}_{d}(\Z[\zeta_p]).$ Then the representation can be reduced modulo a power of $h,$ where $h=1-\zeta_p$ is an irreducible in the cyclotomic ring $\Z[\zeta_p],$ and satisfies $\Z[\zeta_p]/(h)\simeq \F_p.$

It turns out that the representations $\rho_{p,k}=\rho_p \ \mathrm{mod} \ h^k$ have some nice compatibility with the Johnson filtration of the mapping class group. For $k\geq 1,$ let $J_k(\Sigma)$ be the $k$-th Johnson subgroup of $\Mod(\Sigma),$ with the convention that $J_1(\Sigma)$ is the Torelli subgroup of $\Mod(\Sigma),$ and $J_2(\Sigma)$ is the   Johnson subgroup generated by Dehn twists along separating curves (here $g\geq 3$). It is proved in \cite{GM07} that $\Ker \rho_{p,1}$ contains $J_2(\Sigma).$ 

Taking advantage of this result and the structure of the abelianization of the Torelli subgroup $J_1(\Sigma)$ computed by Johnson \cite{Joh3}, we can completely describe $\Ker \rho_{p,1}$.
\begin{theorem}\label{thm:main-thm1}
Let $p\geq 5$ be a prime, $\Sigma$ a closed surface of genus $g\geq 3,$ and $\rho_p$ be the $\mathrm{SO}(3)$-WRT quantum representation at level $p.$ Then 
$$\mathrm{Ker}(\rho_{p,1}) = [J_1(\Sigma),J_1(\Sigma)] T_p,$$
where $J_1(\Sigma)$ is the Torelli subgroup of $\Mod(\Sigma)$ and $T_p$ is the subgroup generated by $p$-th powers of Dehn twists. In particular the kernel of $\rho_p$ is contained in $[J_1(\Sigma),J_1(\Sigma)] T_p$.
\end{theorem}
We can obtain a similar result for quantum representations of surface groups. Let $\Sigma$  be a surface of genus at least two with one boundary component and let $\hat{\Sigma}$ be the surface obtained by gluying a disc on the boundary component of $\Sigma$. We can look at the boundary pushing subgroup of $\Mod(\Sigma)$. This group is the kernel of the map $\Mod(\Sigma)\to \Mod(\hat{\Sigma})$  and is naturally isomorphic to a central extension of the fundamental group of $\hat{\Sigma}$. Following \cite{KS16}, the restriction of $\rho_p$ to the boundary pushing subgroup of $\Mod(\Sigma)$ gives a projective representation : 
$$\rho_p^s: \pi_1(\hat{\Sigma}) \longrightarrow \mathrm{PGL}_{d}( \Z[\zeta_p])$$  
\begin{theorem} \label{thm:main-thm1.2}
Let $\Sigma$  be a surface of genus at least two with one boundary and $p \ge 5$ be prime. Suppose that the boundary of $\Sigma$ is colored by $2$ then the kernel of $\rho^s_{p,1}: \pi_1(\hat{\Sigma}) \to \mathrm{PGL}_{d}( \mathbb{F}_p)$ is $$\Ker(\pi_1(\hat{\Sigma}) \to H_1(\hat{\Sigma},\mathbb{F}_p))$$ which is the kernel of the mod $p$ abelianization of $\pi_1(\hat{\Sigma})$.
\end{theorem}

Our last result concerns the kernel of $\rho_{p,2}$. In \cite{GM14}, a decomposition of the representations $\rho_{p,1}$ as the sum of two irreducible representations of $\Mod(\Sigma)$ over $\F_p$ is introduced (which is called the \textit{odd-even decomposition} by Gilmer and Masbaum). Using the structure of the abelianization of the Johnson subgroup $J_2(\Sigma)$ for surfaces $\Sigma$ of genus $g\geq 6,$ and the odd-even decomposition of $\rho_p,$ we prove:
\begin{theorem}\label{thm:main-thm2}
Let $\Sigma$ a closed surface of genus $g\geq 6.$ Then for all large enough primes $p$ we have: 
$$\mathrm{Ker}(\rho_{p,2}) = [J_2(\Sigma),J_2(\Sigma)] T_p,$$
where $J_2(\Sigma)$ is the Johnson subgroup of $\Mod(\Sigma)$ and $T_p$ is the subgroup generated by $p$-th powers of Dehn twists. In particular, $\mathrm{Ker}(\rho_{p}) \subset [J_2(\Sigma),J_2(\Sigma)] T_p,$.
\end{theorem}
As another ingredient in the proof of Theorem \ref{thm:main-thm1} and \ref{thm:main-thm2}, we prove some irreducibility  results for some modular representations of $\Sp_{2g}(\F_p),$ in order to identify the image of the representations $\rho_{p,1}|_{J_1(\Sigma)}$ and $\rho_{p,2}|_{J_2(\Sigma)}$ with the mod $p$ abelianization of $J_1(\Sigma)$ and $J_2(\Sigma)$ respectively.

\textbf{Acknowledgements:} Over the course of this work, the first author was partially supported by the projects AlMaRe (ANR-19-CE40-0001-01) and by the project "CLICQ" of the R\'egion Bourgogne Franche Comt\'e. The authors thank Gwenael Massuyeau for helpful conversations. We will also thank an anonymous referee for helping us with improving the exposition and for pointing out a gap in the proof of Proposition \ref{prop:image_modh2} in an earlier version of the paper.
\section{Preliminaries}
\label{sec:prelim}
\subsection{Basic properties of $\mathrm{SO}(3)$-WRT representations}
\label{sec:basics_quantumreps}
In this section, we will sketch a construction of the representation $\rho_p,$ and state some of its properties. Let $p$ be an odd integer $\geq 3,$ let $M$ be a closed compact oriented $3$-manifold, and let $L$ be a framed link in $M.$ In \cite{BHMV}, a topological invariant $Z_p(M,L)\in \Q[\zeta_p]$ of the pair $(M,L)$ is defined. Roughly speaking, $Z_p(M,L)$ is the evaluation at a primitive $p$-th root of unity $\zeta_p$ of a suitable linear combination of the colored Jones polynomials of the link $L_0 \cup L,$ where $L_0$ is a surgery presentation of $M.$ Here suitable means that the invariant thus defined is independent of the choice of surgery presentation $L_0$ for $M,$ this is achieved by coloring the surgery presentation $L_0$ by the so-called \textit{Kirby color}. We refer to \cite{BHMV} for all details on this construction; we will use only some of the properties of this invariant $Z_p(M,L).$
First we note that as a consequence of the construction, the invariant $Z_p(M,L)$ satisfies the Kauffman relations in terms of the link $L.$ Moreover, by a theorem of Masbaum and Roberts \cite{MR98}, when $p$ is prime then $Z_p(M,L) \in \Z[\zeta_p].$ 

Let now $\Sigma$ be a closed compact oriented surface. For $M$ and $M'$ two $3$-manifolds with a fixed identification of their boundaries $\partial M, \partial M'$ with $\Sigma,$ we write $\langle M,M'\rangle= Z_p(M \underset{\Sigma}{\cup} \overline{M'}).$ We can do the same construction for manifolds containing framed links $(M,L),(M',L').$ Moreover, we can extend $\langle ,\rangle$ by bilinearity to get a sesquilinear form $\langle , \rangle$ on the $\Q[\zeta_p]$-vector space $\mathcal{V}_p(\Sigma)$ formally spanned by $3$-manifold $M$ with a fixed isomorphism $\partial M=\Sigma,$ and containing a link $L.$

As examples of elements in $\mathcal{V}_p(\Sigma),$ consider $H$ a handlebody with boundary $\Sigma,$ and such that $H$ is itself a tubular neighborhood of a banded trivalent graph $\Gamma.$ A $p$-admissible coloring of the trivalent graph associates an non-negative even integer to each edge, with the additional conditions that colorings near a vertex satisfy triangle inequalities $c_1\leq c_2+c_3$ and also $c_1+c_2+c_3\leq 2p-4.$
Given a $p$-admissible coloring $c$ of the trivalent graph $\Gamma,$ we get an element of $\mathcal{V}_p(\Sigma)$ as $(H,\Gamma(c)),$ where $\Gamma(c)$ is the linear combination of links obtained by cabling each edge $e$ of $\Gamma$ by the Jones-Wenzl idempotent $f_{c_e}.$
\begin{theorem}\label{thm:BHMV}\cite{BHMV}For any odd integer $p\geq 3,$ The vector space 
$$V_p(\Sigma)=\mathcal{V}_p(\Sigma) / \Ker \langle ,\rangle$$ is a $\Q[\zeta_p]$-vector space of finite dimension $d(g,p),$ on which $\Mod(\Sigma)$ has a natural projective-linear action by change of identification $\partial M=\Sigma.$

Moreover, for any handlebody $H$ with $\partial H=\Sigma,$ and any pants decomposition of $\Sigma$ by curves which bound disks in $H,$ colored trivalent graphs in $H$ with $p$-admissible colors form a basis of $V_p(\Sigma).$
\end{theorem}
An explicit formula for $\dim(V_p(\Sigma))$ may be found in \cite{BHMV}.
Now, if we suppose that $p$ is prime, we can instead define $\mathcal{V}_p(\Sigma)$ to be the free $\Z[\zeta_p]$-module spanned by pairs $(M,L)$ where $M$ is $3$-manifold with a fixed identification $\partial M=\Sigma,$ and we get that $\mathcal{S}_p(\Sigma)=\mathcal{V}_p(\Sigma)/\Ker \langle\, ,\rangle$ is a free $\Z[\zeta_p]$-module and a lattice in $V_p(\Sigma)$, moreover $\Mod(\Sigma)$ still admits a natural action on it. 

The previous basis of $V_p(\Sigma)$ as a $\Q[\zeta_p]$-vector space does not provide basis of $\mathcal{S}_p(\Sigma)$ as a $\Z[\zeta_p]$-module. In \cite{GM07}, Gilmer and Masbaum provide integral basis for $\mathcal{S}_p(\Sigma).$ Unlike the previous basis, they can only be associated to some special pants decomposition of $\Sigma,$ the so-called \textit{lollipop tree} decompositions. We will say that a pants decomposition of $\Sigma$ is a lollipop tree decomposition if it contains $g$ non separating curves and $2g-3$ separating curves.

Given a lollipop tree decomposition of $\Sigma,$ we associate elements $v(a,b)\in \mathcal{S}_p(\Sigma)$ to colorings of the edges of the trivalent graph $\Gamma$ associated to the decomposition as follows. We take $H$ to be a handlebody with $\partial H=\Sigma,$ and such that the curve of the pants decomposition bound disks in $H.$ We let the internal edges of $\Gamma$ be colored by $(2a_i)_{1\leq i \leq 2g-3},$ and we assume that $2a_1,\ldots,2a_g$ are the colors of the edges that are adjacent to one-edge loops. The loop edges are colored by $a_i+b_i.$ We assume again that if $c_1,c_2,c_3$ are the colors around a trivalent vertex, then we have the $p$-admissibility conditions: $c_1\leq c_2+c_3,$ and $c_1+c_2+c_3$ is even and $\leq 2p-4.$ Moreover, we ask for the colors $a_i+b_i$ to be at most $(p-3)/2,$ which Gilmer and Masbaum call a \textit{small} coloring. We construct an element $\Gamma(a,b) \in \mathcal{S}_p(\Sigma)$ by embedding the trivalent graph $\Gamma$ in $H,$ replacing the internal edges with $2a_i$ parallel copies with the Jones-Wenzl idempotent $f_{2a_i}$ inserted. However, for the loop edge colored by $a_i+b_i,$ we first cable the loop edge by the Jones-Wenzl idempotent $f_{a_i},$ then we add a copy of the loop colored by $(\frac{2+z}{h})^{b_i},$ where $h=1-\zeta_p.$ (Here we use the convention that a framed link colored by $z^n$ stands for $n$ parallel copies of the link, and we can make sense of framed links colored by a polynomial in $z$ by extending linearly). Finally, we define $v(a,b)=h^{-\lfloor \frac{1}{2}(a_1+\ldots+a_g)\rfloor}\Gamma(a,b).$
\begin{theorem}\label{thm:integral_basis}\cite{GM07} 
Let $\Sigma$ be a closed orientable surface and fix a lollipop tree decomposition of $\Sigma.$ Then the vectors $v(a,b)$ where $(a,b)$ runs over all small $p$-admissible colorings, form a $\Z[\zeta_p]$-basis of $\mathcal{S}_p(\Sigma)$ as a $\Z[\zeta_p]$-module.
\end{theorem}

In \cite{GM14}, Gilmer and Masbaum introduce a decomposition of $\mathcal{S}_p(\Sigma)$ as a sum of two $\Z[\zeta_p]$-submodules. Let $v(a_i,b_i)$ be a lollipop tree basis of $\mathcal{S}_p(\Sigma),$ and assume  (up to reindexing) that $2a_1,\ldots 2a_g$ are the colors of the edges that are adjacent to a one-edge loop. They define $\mathcal{S}_p^{odd}(\Sigma)$ (resp. $\mathcal{S}_p^{ev}(\Sigma)$) to be the $\Z[\zeta_p]$-submodule spanned by the basis vectors $v(a_i,b_i)$ such that $a_1+\ldots +a_g$ is odd (resp. even). 

While this decomposition of $\mathcal{S}_p(\Sigma)$ depends on the choice of a lollipop tree basis of $\mathcal{S}_p(\Sigma),$ Gilmer and Masbaum show that the associated decomposition of the $\Z[\zeta_p]/(h)\simeq \F_p$-module 
$$F_p(\Sigma):=\mathcal{S}_p(\Sigma)\underset{\Z[\zeta_p]}{\otimes}\Z[\zeta_p]/(h)$$ does not depend on such a choice.
\begin{theorem}\label{thm:decomp}\cite{GM14}Let $\Sigma$ be a closed compact oriented surface and $p\geq 5$ be a prime. 

Then we have 
$$\Z[\zeta_p][\rho_p(\Mod(\Sigma))] = \begin{pmatrix}
\mathrm{End}(\mathcal{S}_p^{odd}(\Sigma)) & \mathrm{Hom}(\mathcal{S}_p^{ev}(\Sigma),\mathcal{S}_p^{odd}(\Sigma))
\\ h\mathrm{Hom}(\mathcal{S}_p^{odd}(\Sigma),\mathcal{S}_p^{ev}(\Sigma)) & \mathrm{End}(\mathcal{S}_p^{ev}(\Sigma))
\end{pmatrix}$$
where $h=1-\zeta_p.$
\end{theorem}
In these results $\Mod(\Sigma)$ really means the central extension of $\Mod(\Sigma)$ that acts linearly rather than only projectively. In particular we can talk here about $\mathrm{Hom}$ and $\mathrm{End}$. The kernel of this extension is a cyclic group of order $p$ and the elements of this kernel act as multiplication by a power of $\zeta_p$.

This central extension of $\Mod(\Sigma)$ is trivial when restricted to the Torelli group as it is the pull-back of an extension of the symplectic group. The above theorem implies that the image of the Torelli and Johnson subgroups by $\rho_p$ have the following structure:
\begin{corollary}\label{cor:decompJohnson}

(i) If $f\in J_1(\Sigma),$ then
$$\rho_p(f)=\begin{pmatrix}
id_{\mathcal{S}_p^{odd}(\Sigma)} & A
\\ 0 & id_{\mathcal{S}_p^{ev}(\Sigma)}
\end{pmatrix} \ (\mathrm{mod} \ h)$$
for some $A\in \mathrm{Hom}(\mathcal{S}_p^{ev}(\Sigma),\mathcal{S}_p^{odd}(\Sigma)).$

(ii) If $g\in J_2(\Sigma),$ then
$$\rho_p(g)=id_{\mathcal{S}_p(\Sigma)} + h \begin{pmatrix}
A_1 & A_2 \\ 0 & A_3
\end{pmatrix} \ (\mathrm{mod} \ h^2)$$
for some $A_1\in \mathrm{End}(\mathcal{S}_p^{odd}(\Sigma)),$ $A_2 \in \mathrm{Hom}(\mathcal{S}_p^{ev}(\Sigma),\mathcal{S}_p^{odd}(\Sigma))$ and $A_3 \in \mathrm{End}(\mathcal{S}_p^{ev}(\Sigma)).$

(iii) If $f\in J_1(\Sigma)$ and $g\in J_2(\Sigma),$ then 
$$\rho_p([f,g])=id_{S_p(\Sigma)}+h\begin{pmatrix}
	0 & B \\ 0 & 0 
	\end{pmatrix} \ (\mathrm{mod} \ h^2),$$
	for some $B\in \mathrm{Hom}(\mathcal{S}_p^{ev}(\Sigma),\mathcal{S}_p^{odd}(\Sigma)).$
\end{corollary}
\begin{proof}
The first part of the corollary is proved in \cite[Lemma A1]{GM17}.
For the second part, notice that it is trivially true for $f=t_{\alpha},$ a Dehn twist along a separating curve of the lollipop tree pants decomposition. Note that there is a Dehn twist of each possible genus among those. Conjugating with an element of $\Mod(\Sigma)$ and by Theorem \ref{thm:decomp}, it is then also true for any separating Dehn twist, and therefore for any $f\in J_2(\Sigma)$ as $J_2(\Sigma)$ is generated by separating twists.
Finally, point (iii) follows from point (i) and (ii), as:
$$\rho([f,g])=id_{S_p(\Sigma)}+h\begin{pmatrix}
0 & AA_3-A_1A \\ 0 & 0
\end{pmatrix} \ (\mathrm{mod} \ h^2).$$
\end{proof}

\subsection{Computations of $\rho_p$ on bounding pairs and separating twists}
\label{sec:bounding_pair}

 The goal of this subsection is to prove that certain bounding pairs act non trivially via $\rho_{p,1},$ and that the images of some specific familly of separating twists by $\rho_{p,2}$ are linearly independent. These results are technical (based on skein theoretical arguments) but are crucial for the proofs of Theorem \ref{thm:main-thm1} and \ref{thm:main-thm1.2}.

Let $p \ge 5$ be a prime and $A$ be a $2p$-th primitive root of unit. Cutting the surface into smaller pieces is a usual trick used for computations in the setting of quantum representations. 

Let $B$ be a $3$-ball, let $\mathcal{P}$ be a set of four banded points on $\partial B$ where two of them are colored by $1$ and two are colored by $2$. Let $\gamma \subset \partial B$ be the following simple closed curve : 
$$\includegraphics[width=1.5in]{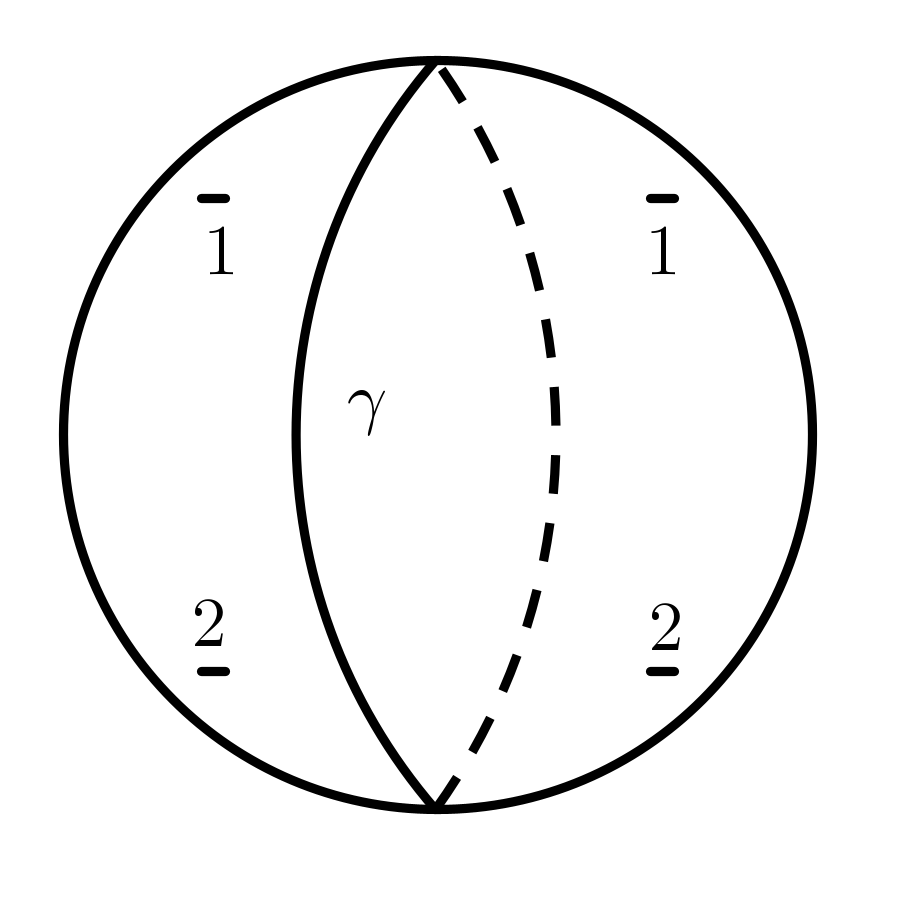}$$
Let $V$ be the skein module over $\Q[A]$ of $B$ relative to $(\partial B, \mathcal{P})$. This $\Q[A]$ vector space is two dimensional with basis $$e_1 = \begin{minipage}{0.61in}\includegraphics[width=0.6in]{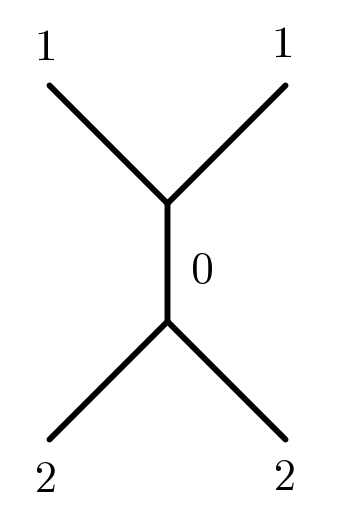} \end{minipage} \quad  \text{and} \quad e_2 = \begin{minipage}{0.61in}\includegraphics[width=0.6in]{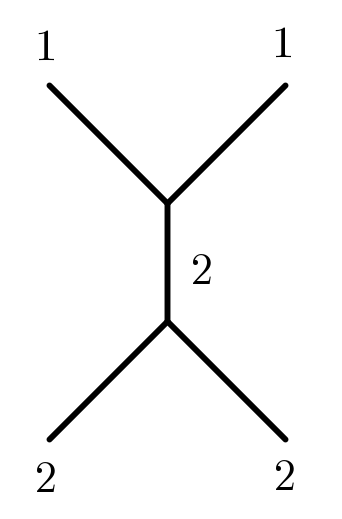} \end{minipage} $$
Recall that for $k$ a positive integer, $[k] = \dfrac{A^{2k}-A^{-2k}}{A^{2}-A^{-2}}$.
\begin{lemma}The action on $V$ of the Dehn twist along $\gamma$ has the following matrix in the basis $(e_1,e_2)$ :
$$\begin{pmatrix}
T_{1,1} & T_{1,2} \\
T_{2,1} & T_{2,2}
\end{pmatrix} = \begin{pmatrix}
-\dfrac{A^3+A^{19}+A^{11}}{[3]} & -\dfrac{A^9(A^4+A^{-4})(A^2-A^{-2})}{A^2+A^{-2}} \\
A^9(A^4-A^{-4}) & -\dfrac{A^7+A^{-1}+A^{15}}{[3]}
\end{pmatrix}$$ 
\end{lemma}

\begin{proof}
This is obtained by straightforward skein theoretical computations.
\end{proof}



\begin{lemma}\label{lemma:bounding_pair} Let $\Sigma$ be a closed surface of genus $g\geq 3,$ then the image of a bounding pair of genus $1$ by $\rho_{p,1}$ is not trivial.
\end{lemma}
\begin{proof}
We note that since all bounding pairs of genus $1$ are conjugated, it suffices to prove the proposition for one particular bounding pair.
Let $\Gamma$ be the following graph: 
$$\includegraphics[width=2.5in]{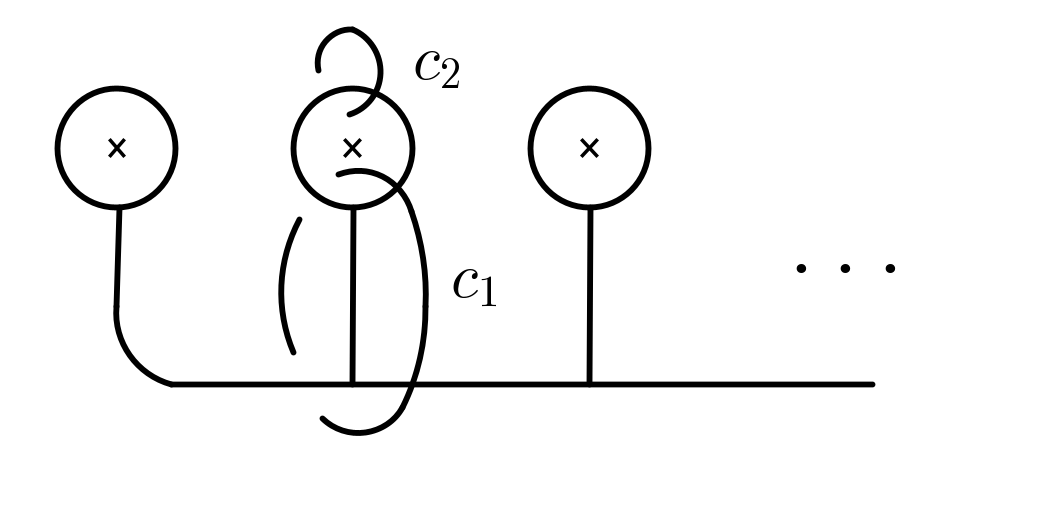}$$
 Here $\Sigma$ is the boundary of a regular neighborhood of the graph, which is a genus $g$ handlebody and the curves $c_1, c_2$ are on $\Sigma$.  
We want to compute the action of $t_{c_1} t_ {c_2}^{-1},$ which is a bounding pair of genus $1.$ We define the following two vectors associated to an admissible coloring of $\Gamma$ : 

$u = \begin{minipage}{2.5in}\includegraphics[width=2.5in]{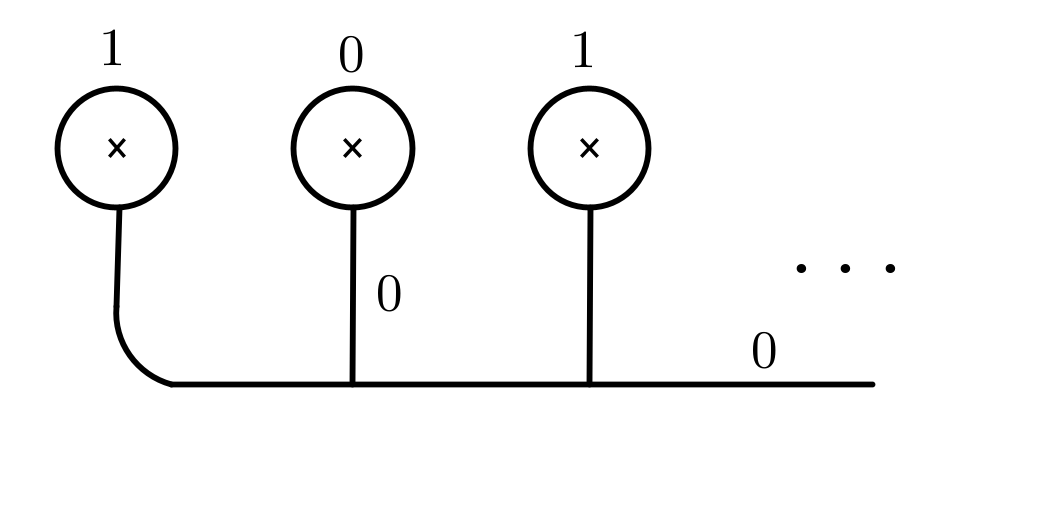} \end{minipage}  \text{and} \quad v = \begin{minipage}{0.5in}\includegraphics[width=2.5in]{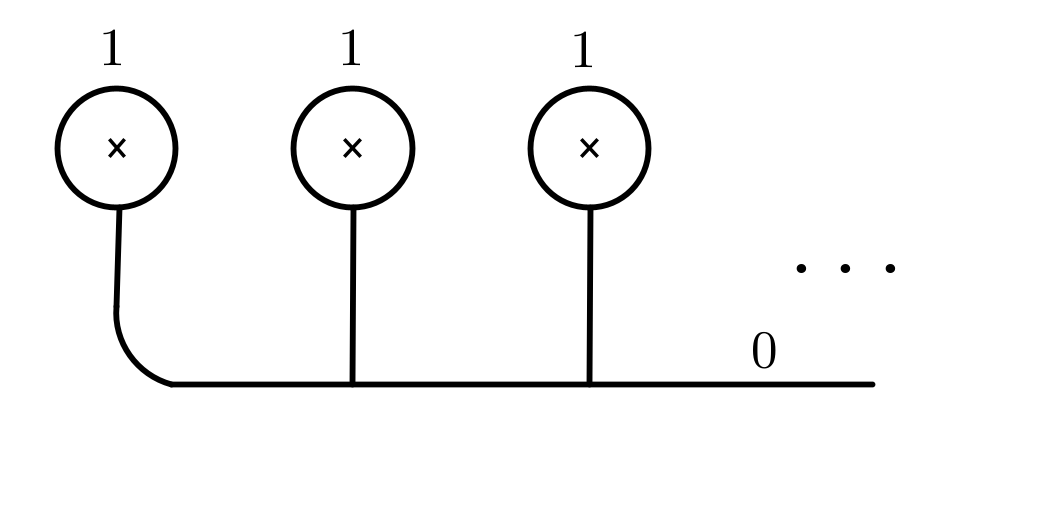} \end{minipage}$
\vspace{0.5cm}

\noindent here the three dots means that the remaining colors are $0$. The triple 
 $$(h^{-1} u, h^{-1} v,h^{-2}(2+z_2)u)$$ is part of the integral basis. Here $z_2$ is the following curve
$$\includegraphics[width=2.8in]{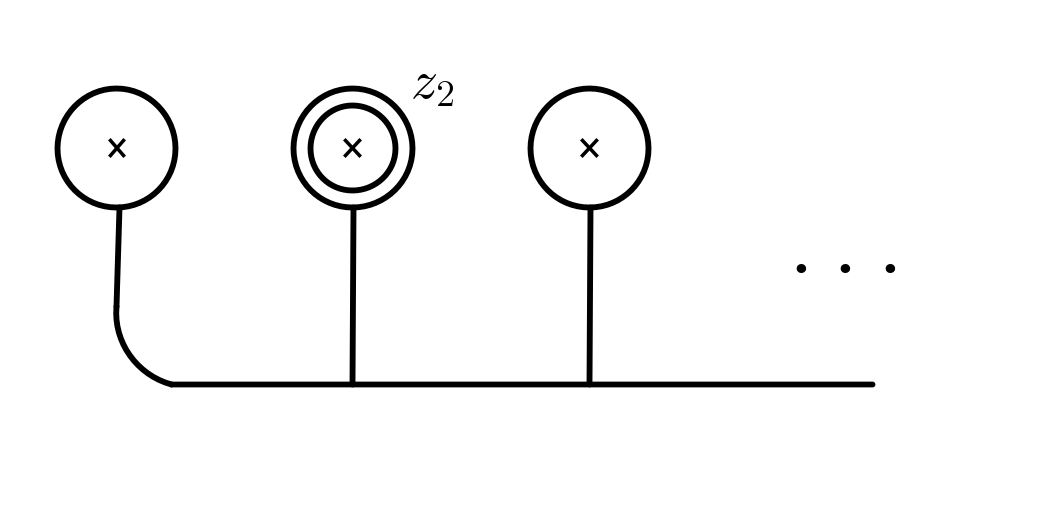}$$
Let $W$ be the $\Z[\zeta_p]$-module generated by $\{h^{-1} u, h^{-1} v,h^{-2}(2+z_2)u \}$. It is easy to check that $\rho_p(t_{c_1})$ and $\rho_p(t_{c_2}^{-1})$ stabilize $W$. On the basis $(h^{-1} u, h^{-1} v,h^{-2}(2+z_2)u)$ of $W$, we compute 

$$\rho_p(t_{c_1})_{|_W}=\begin{pmatrix}
A^8 & -2T_{1,2} & 2h^{-1}(A^8-T_{1,1})\\
0 & T_{2,2} & h^{-1}T_{2,1} \\
0& hT_{1,2} & T_{1,1}
\end{pmatrix} \, , \, \rho_p(t_{c_2})_{|_W}=\begin{pmatrix}
1 & 0 & 2h^{-1}(1+A^3)\\
0 & - A^3 & 0 \\
0& 0 & -A^3
\end{pmatrix}$$ Recall that $A^2 = \zeta_p$ and $h=1-\zeta_p$, we have that 

$$\rho_p(t_{c_1} t_{c_2}^{-1})_{|_W} \equiv \begin{pmatrix}
1 & 0 & 0\\
0 & 1 & -4 \\
0& 0 & 1
\end{pmatrix} \pmod h$$
\end{proof}
Next we prove a linear independence result for the $\rho_{p,2}$ image of some separating twists. Recall that $\mathrm{Im}\rho_{p,2}|_{J_2(\Sigma)}$ has a natural structure of $\F_p$-vector space, obtained by considering the order $1$ in $h$ (see also Lemma \ref{lemma:modh2_abelian}).
\begin{lemma}
	\label{lemma:independence}Let $p\geq 5$ be a prime. Let $\mathcal{C}$ be a lollipop pants decomposition of $\Sigma,$ let $v(a,b)$ be the associated basis of $\mathcal{S}_p(\Sigma),$ and let $\alpha_1,\ldots,\alpha_{2g-3}$ be the separating curves of $\mathcal{C}.$ Then the $\rho_{p,2}(t_{\alpha_i})$ are $\F_p$-linearly independent.
\end{lemma}
\begin{proof}
	In the basis $v(a,b),$ we will denote by $2a_i$ the color corresponding to the curve $\alpha_i.$ For $1\leq i \leq 2g-3,$ let $v_i$ be the vector $v(a,0),$ where $a_i=2$ and $a_j=1$ for $j\neq i.$ Let also $w$ be the vector $v(a,0),$ where $a_i=1$ for all $i.$
	
	The separating Dehn twists $t_{\alpha_i}$ acts diagonally on the basis $v(a,b),$ and moreover we have:
	$$\rho_{p}(t_{\alpha_i})(v(a,b))=\zeta_p^{2a_i(a_i+1)}v(a,b),$$
	by \cite[Remark 7.6(ii)]{BHMV}, hence we have
	$$\rho_{p,2}(t_{\alpha_i})(v_j)=\begin{cases}
	(1+12h)v_i \ (\mathrm{mod} \ h^2) \ \textrm{if }i=j
	\\ (1+4h)v_j \ (\mathrm{mod} \ h^2) \ \textrm{else}
	\end{cases}$$
	We conclude as the matrix 
	$$4\begin{pmatrix}
	3 & 1 & \ldots & 1
	\\ 1 & \ddots & \ddots & \vdots
	\\ \vdots & \ddots & \ddots & 1
	\\ 1 & \ldots & 1 & 3
	\end{pmatrix}$$
	has determinant $4^{2g-3}2^{2g-4}(2g-1),$ which is non-zero mod $p$ unless $p$ divides $2g-1.$ In the latter case its kernel is generated by ${}^t(1,1,\ldots,1).$ However, in that case we also compute $\rho_{p,2}(t_{\alpha_1}\ldots t_{\alpha_{2g-3}})(w)=(1+4(2g-3)h)w \ (\mathrm{mod} \ h^2),$ hence the $\rho_{p,2}(t_{\alpha_i})$ are still linearly independent.
	
\end{proof}

Recall that if $\Sigma$ is a surface of genus at least two with one boundary component, we denote by $\hat{\Sigma}$ the surface obtained by gluing a disk on the boundary of $\Sigma$. We also denote by $\rho_p^s$ the representation of $\pi_1(\hat{\Sigma})$ obtained by restriction of $\rho_p$ where the boundary color is $2$.
\begin{lemma} \label{lemma:not_pi1_trivial}
If $\gamma \in \pi_1(\hat{\Sigma})$ is freely homotopic to a non separating simple closed curve then $\rho_{p,1}^s(\gamma)$ is not trivial.
\end{lemma}
\begin{proof}
If $\gamma$ is a loop in $\pi_1(\hat{\Sigma})$ and $\varphi \in \Mod(\Sigma)$ then $$\rho_p(\varphi) \rho_p(\gamma) \rho_p(\varphi)^{-1} = \rho_p(\varphi(\gamma))$$
Moreover the action of $\Mod(\Sigma)$ is transitive on the set of non separating simple loops in $\pi_1(\hat{\Sigma})$. Therefore it is enough to find a simple loop $\gamma \in \pi_1(\hat{\Sigma})$ such that $\rho_{p,1}^s(\gamma) \neq 1 $ Let $\Gamma$ be the following graph 
$$\includegraphics[width=2.2in]{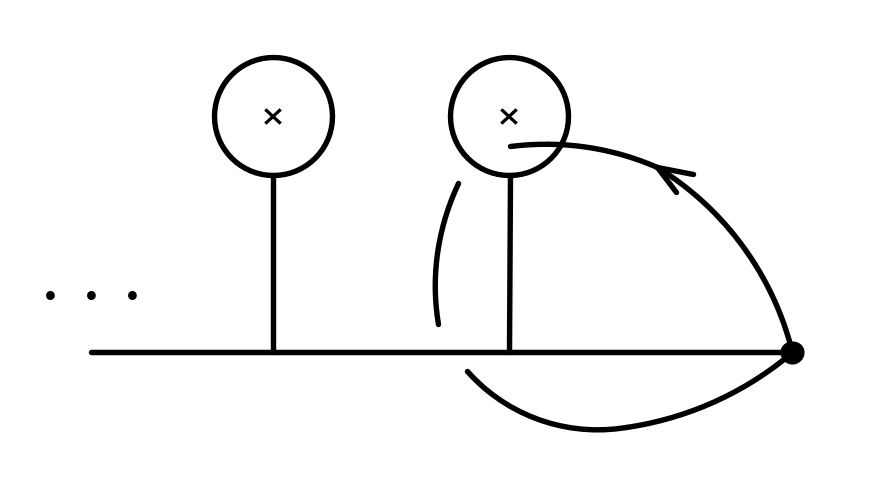}$$
As before $\hat{\Sigma}$ is the boundary of a regular neighborhood of the graph, the univalent vertex (marked by a dot) is attached to the banded point colored by $2$ on $\hat{\Sigma}$. Also the loop $\gamma$ lies on $\hat{\Sigma}$. Let $c_1$ and $c_2$ be the two following curves on $\Sigma$ 
$$\includegraphics[width=2.2in]{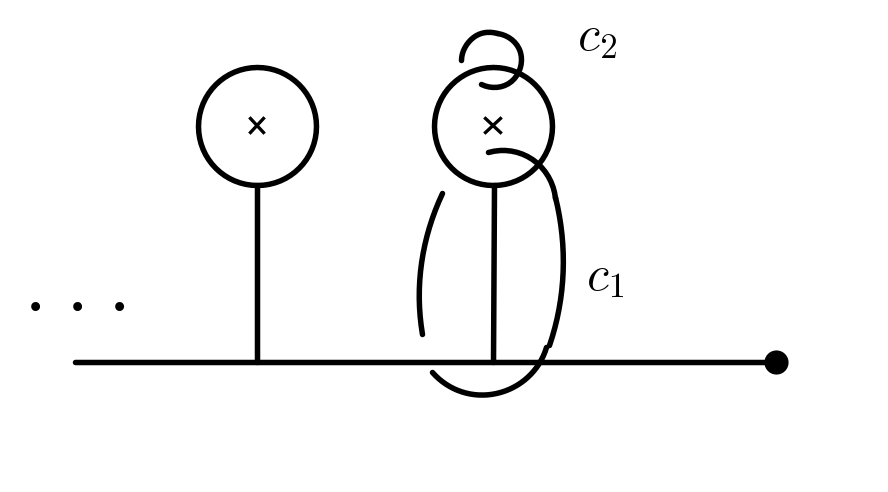}$$
It is known that the loop $\gamma$ is $t_{c_1} t_{c_2}^{-1}$ when viewed as an element of $\Mod(\Sigma)$. We will omit the computations, as they are almost identical to the ones done for the proof of Lemma \ref{lemma:bounding_pair}, but these computations show that $\rho_{p,1}(t_{c_1} t_{c_2}^{-1})$ is not trivial.

\end{proof}
\begin{lemma}\label{lemma:not_J1_invariant}Let $\Sigma$ be a surface of genus $g\geq 3,$ and let $\alpha$ be a separating curve on $\Sigma$ of genus $1.$ Then there exists $f\in J_1(\Sigma)$ such that $\rho_{p,2}(f t_{\alpha} f^{-1})\neq \rho_{p,2}(t_{\alpha}),$ where $t_{\alpha}$ is the Dehn twist along $\alpha.$
\end{lemma}
\begin{proof}
We will use the same notation as in the proof of Lemma \ref{lemma:bounding_pair}. Let $\alpha$ be the meridian curve of the base edge of the second lollipop of the graph $\Gamma.$ Let also $f=t_{c_1}t_{c_2}^{-1},$ the bounding pair considered in the proof of Lemma \ref{lemma:bounding_pair}. We have that the $\Z[\zeta_p]$-submodule $W$ of $\mathcal{S}_p(\Sigma)$ spanned by $h^{-1}u,h^{-1}v,h^{-2}(2+z_2)u$ is invariant not only by $f=t_{c_1}t_{c_2}^{-1}$ but also by $t_{\alpha}.$ Indeed, in the basis $\lbrace h^{-1}u,h^{-1}v,h^{-2}(2+z_2)u \rbrace,$ we have 
$$\rho_p(t_{\alpha})|_W=\begin{pmatrix}
1 & 0 & 0 \\ 0 & A^8 & 0 \\ 0 & 0 & 1
\end{pmatrix}, \ \textrm{thus} \ \rho_{p,2}(t_{\alpha})|_W=\begin{pmatrix}
1 & 0 & 0 \\ 0 & 1-4h & 0 \\ 0 & 0 & 1
\end{pmatrix} \ (\mathrm{mod} \ h^2).$$
Note that since $\rho_p(t_\alpha)=Id \ \mod \ h,$ it suffices to know $\rho_p(f)$ mod $h$ to compute $\rho_p(ft_{\alpha}f^{-1})$ mod $h^2.$ We get
$$\rho_{p,2}(ft_{\alpha}f^{-1})=\begin{pmatrix}
1 & 0 & 0 \\ 0 & 1 & -4 \\ 0 & 0 & 1
\end{pmatrix}\begin{pmatrix}
1 & 0 & 0 \\ 0 & 1-4h & 0 \\ 0 & 0 & 1
\end{pmatrix}\begin{pmatrix}
1 & 0 & 0 \\ 0 & 1 & 4 \\ 0 & 0 & 1
\end{pmatrix}=\begin{pmatrix}
1 & 0 & 0 \\ 0 & 1-4h & -16h \\ 0 & 0 & 1
\end{pmatrix}\neq \rho_{p,2}(t_{\alpha}).$$
\end{proof}
We will call a subsurface $S\subset \Sigma$ \emph{essential} if no boundary component of $S$ bounds a disk.

For $f \in \Mod(\Sigma)$ and $S$ a subsurface of $\Sigma,$ we say that the \textit{support} of $f$ is included in $S$ if there is a representative of $f$ which is the identity on $\Sigma \setminus S.$ The mapping class $f$ can be then be seen as a mapping class of the surface $S'$ obtained from $S$ by filling boundary components with disks. The surface $S'$ may have smaller genus than $\Sigma.$ We will write $\rho_p$ for the $\mathrm{SO}(3)$-WRT quantum representation of $\Mod(\Sigma)$ or of $\Mod(S'),$ indifferently. 
\begin{lemma}\label{lemma:support} Let $\Sigma$ be a closed compact oriented surface and let $f\in \Mod(\Sigma),$ such that the support of $f$ is contained in a essential subsurface $S\subset \Sigma.$ Let $S'$ be the closed surface obtained from $S$ by filling each boundary component with a disk, and let $f'\in \Mod(S')$ be the mapping class induced by $f.$ Then, for any odd $p\geq 5,$ we have 
\begin{itemize}
\item[(i)]$\rho_p(f')\notin \Ker \rho_p \Longrightarrow \rho_p(f) \notin \Ker \rho_p.$
\item[(ii)]Let $J$ be an ideal of $\Z[\zeta_p].$ If furthermore $\partial S$ consists only of separating curves in $\Sigma,$ then 
$$\rho_p(f') \neq Id \ \mathrm{mod} \ J \Longrightarrow \rho_p(f) \neq Id \ \mathrm{mod} \ J.$$
\end{itemize}
\end{lemma}
\begin{proof} The lemma is a direct consequence of the description of the basis of $V_p(\Sigma)$ as a $\Q[\zeta_p]$-vector space in Theorem \ref{thm:BHMV} and the integral basis of $\mathcal{S}_p(\Sigma)$ as a $\Z[\zeta_p]$-module in Theorem \ref{thm:integral_basis}. In both case, we get that $\rho_p(f)$ has a block conjugated to $\rho_p(f'),$ and thus is non trivial. Indeed, let us choose a pants decomposition of $\Sigma$ (or lollipop tree pants decomposition of $\Sigma$) containing the boundary components of $S$ as pants decomposition curves. Let $\lbrace\Gamma(a) \rbrace$ or $\lbrace v(a,b) \rbrace$ be the associated $\Q[\zeta_p]$- or $\Z[\zeta_p]$-basis of $V_p(\Sigma)$ or $\mathcal{S}_p(\Sigma).$  The subspace of $V_p(\Sigma)$ (resp. $\mathcal{S}_p(\Sigma)$) spanned by vectors $\Gamma(a)$ (resp. $v(a,b)$) such that the colors of edges not belonging to $S$ are identically zero is isomorphic to $V_p(S')$ (resp. $\mathcal{S}_p(S')$) stable by $\rho_p(f),$ and the corresponding block is conjugated to $\rho_p(f').$
\end{proof}
\subsection{Proof of Theorem \ref{thm:main-thm1.2}}
In this short subsection we give a proof of Theorem \ref{thm:main-thm1.2}. Let $\Sigma$  be a surface of genus at least two with one boundary component and $p \ge 5$ be prime. Suppose that the boundary of $\Sigma$ is colored by $2$, we want to understand the kernel of $\rho^s_{p,1}: \pi_1(\hat{\Sigma}) \to \mathrm{PGL}_{d}( \mathbb{F}_p)$. 

Recall that any Dehn twist along a separating curve is trivial via $\rho_{p,1}$. Now any separating simple loop $\delta \in \pi_1(\hat{\Sigma})$, when viewed in $\Mod(\Sigma)$, can be written as $t_{c_1} t_{c_2}^{-1}$ where $c_1,c_2$ are separating curves. Therefore the image by $\rho^s_{p,1}$ of any separating simple loop is trivial. 

By \cite[Lemma A.1]{P07}, the group generated by separating simple loops on $\hat{\Sigma}$ is $[\pi_1(\hat{\Sigma}),\pi_1(\hat{\Sigma})]$, so $\rho_{p,1}^s$ factors through $\pi_1(\hat{\Sigma})/[\pi_1(\hat{\Sigma}),\pi_1(\hat{\Sigma})]$. Moreover $\rho_{p,1}^s$ kills the $p$-th powers of each generator of $\pi_1(\hat{\Sigma})$, so $\rho_{p,1}^s$ induces a map :
$$\bar{\rho}_{p,1}^s :  H_1(\hat{\Sigma},\mathbb{F}_p) \to \mathrm{PGL}_{d}( \mathbb{F}_p)$$ Let $I_p$ be the kernel of $\bar{\rho}_{p,1}^s$. To conclude, we need to prove that $I_p=0$. The map $\bar{\rho}_{p,1}^s$ is equivariant with respect to the mapping class group in the sense that if $[\gamma]$ is the homology class of the loop $\gamma$ then 
$$\bar{\rho}_{p,1}^s(\varphi_*[\gamma]) = \rho_{p,1}^s(\varphi) \bar{\rho}_{p,1}^s([\gamma]) \rho_{p,1}^s(\varphi)^{-1}$$ for any $\varphi \in \Mod(\Sigma)$. This property implies that $I_p$ is a $\mathbb{F}_p$-subspace of $H_1(\hat{\Sigma},\mathbb{F}_p)$ invariant under $\Mod(\Sigma)$. The action of $\Mod(\Sigma)$ on $H_1(\hat{\Sigma},\mathbb{F}_p)$ is irreducible and by Lemma \ref{lemma:not_pi1_trivial} the map $\rho_{p,1}^s$ is not trivial, therefore $I_p = 0$. 
\subsection{Structure of the abelianization of the Torelli and Johnson subgroups}
\label{sec:abelianization}
For $G$ a group, let $\Ab(G)=G/[G,G]$ be its abelianization, which we can consider as a $\Z$-module, and let $\Ab_{\Q}(G)=\Ab(G) \underset{\Z}{\otimes}\Q.$

We will first describe the abelianization of the Torelli group $J_1(\Sigma)$ of a closed surface of genus $g \geq 3,$ as described in the following theorem of Johnson \cite{Joh3}. Note that $\Mod(\Sigma)$ acts on $J_1(\Sigma)$ by conjugation. This action induces a $\Mod(\Sigma)$ action on $\Ab(J_1(\Sigma)),$ which factors through an $\Sp_{2g}(\Z)$-action.

We will write $\omega$ for the intersection form on $\Sigma,$ and by abuse of notations, for its dual which is a $\Sp_{2g}(\Z)$-invariant vector in $\Lambda^2 H_1(\Sigma,\Z).$ An explicit formula for $\omega$ is $\omega=a_1\wedge b_1 +\ldots +a_g\wedge b_g,$ where $a_1,b_1,\ldots,a_g,b_g$ is any symplectic basis of $H_1(\Sigma,\Z).$ Similarly, $\omega \wedge H_1(\Sigma,\Z)$ is a subrepresentation of $\Lambda^3 H_1( \Sigma,\Z).$ The \textit{contraction map} $\kappa$ on $\Lambda^3 H_1(\Sigma,\Z)$ is also a surjective morphism of $\Sp_{2g}(\Z)$-representation:
$$\begin{array}{rccl}
\kappa : & \Lambda^3 H_1(\Sigma,\Z) &  \longrightarrow & H_1(\Sigma,\Z)
\\ & a\wedge b \wedge c & \longmapsto & \omega(a,b)c+ \omega(b,c)a + \omega(c,a)b
\end{array}.$$
We note that $\kappa$ maps an element $\omega \wedge c\in \omega \wedge H_1(\Sigma,\Z)$ to $(g-1)c,$ hence $\kappa$ induces a map 
$$\widetilde{\kappa}:\Lambda^3 H_1(\Sigma,\Z)/\left(\omega\wedge H_1(\Sigma,\Z)\right)\longrightarrow H_1(\Sigma,\Z/(g-1)\Z).$$

\begin{theorem}\label{thm:abelianization_Torelli}\cite{Joh3}
For any closed surface of genus $g\geq 3,$ we have an isomorphism of $\Sp_{2g}(\Z)$-modules:
$$\mathrm{Ab}(J_1(\Sigma))\simeq \Lambda^3 H_1(\Sigma,\Z) / \left( \omega \wedge H_1(\Sigma,\Z)\right) \bigoplus T,$$
where $T$ is a $\Sp_{2g}(\Z)$-module which as an abelian group is a finite rank $2$-torsion group. 
\end{theorem}
We note that the map
$$\tau_1 : \mathrm{Ab}(J_1(\Sigma))\longrightarrow \Lambda^3 H_1(\Sigma,\Z) / \left( \omega \wedge H_1(\Sigma,\Z)\right)$$
induced by the above isomorphism is also known as the first Johnson isomrphism, and the right hand-side is also isomorphic to $J_1(\Sigma)/J_2(\Sigma).$

Moreover, the result of Johnson actually explicitely describes the $\Sp_{2g}(\Z)$-module structure of the $2$-torsion group, but we will not need it here; we will only use that it is a $2$-torsion group.

Next we want to describe the abelianization of the Johnson subgroup $J_2(\Sigma).$ As previously, the action of $\Mod(\Sigma)$ by conjugation on $J_2(\Sigma)$ induces a $\mathcal{M}=\Mod(\Sigma)/J_2(\Sigma)$ module structure on $\Ab(J_2(\Sigma)).$ Since $\mathcal{M} \simeq \im \tau_1 \rtimes \Sp_{2g}(\Z),$ we also get a $\Sp_{2g}(\Z)$-module structure on $\Ab(J_2(\Sigma)),$ albeit a non-canonical one. 

At the time of this writing, only the rational abelianization of $J_2(\Sigma)$ is known. It was first computed by Dimca, Hain and Papadima \cite{DHP14} after Dimca and Papadima showed that it was of finite rank\cite{DP13}. The description that we will use here comes from the work of Morita, Sakasai and Suzuki\cite{MSS20}.
\begin{theorem}\label{thm:abelianization_Johnson}\cite{DHP14}\cite{MSS20}
For any closed surface of genus $g\geq 6,$ we have an isomorphism of $\Sp_{2g}(\Z)$-modules:
$$\mathrm{Ab}_{\Q}(J_2(\Sigma)) \simeq \Q \oplus [2^2] \oplus [31^2],$$
where $[2^2]$ and $[31^2]$ stands for the $\Sp_{2g}(\Q)$ representations associated to the Young diagrams $[2^2]$ and $[31^2].$ Moreover, those representations are absolutely irreducible representations of $\Sp_{2g}(\Z).$
\end{theorem}
We recall that a representation over $\Q$ is absolutely irreducible if it is irreducible over $\overline{\Q}.$
We note that just as the $\Sp_{2g}(\Z)$-module structure, this splitting is not canonical.
We will be interested also in the $\mathcal{M}$-module structure of $\Ab_{\Q}(J_2(\Sigma)):$ 

\begin{remark}\label{rk:invarianceOfFactors}
	The structure of the isomorphism of Theorem \ref{thm:abelianization_Johnson} is described in \cite{MSS20}, starting at the paragraph under Remark 7.4, until just before the proof of Theorem 1.4. We gather the following facts from it:
	\begin{itemize}
		\item[(1)] The summand $\Q$ in the above isomorphism is the rational span of the core of the Casson invariant, which by \cite[Theorem 5.7]{Mor} is a morphism
		$$d:J_2(\Sigma)\longrightarrow \Z$$
		which is invariant by the action of conjugation by $\mathrm{Mod}(\Sigma).$ In particular, $d$ vanishes on $[J_1(\Sigma),J_2(\Sigma)].$ 
		\item[(2)] The summand $[2^2]$ is the rational image of the second Johnson homomorphism, which is the map 
		$$\tau_2: J_2(\Sigma) \longrightarrow \left(J_2(\Sigma)/J_3(\Sigma)\right) \otimes \Q.$$
		Therefore, it vanishes on $[J_1(\Sigma),J_2(\Sigma)],$ as $[J_1(\Sigma),J_2(\Sigma)]\subset J_3(\Sigma).$
		\item[(3)] The summand $[31^2]$ is isomorphic to $(J_3/J_4)_{\Q}:=\left(J_3(\Sigma)/J_4(\Sigma)\right)\otimes \Q$ as a $\mathrm{Sp}_{2g}(\Z)$-representation. The projection from $\mathrm{Ab}(J_2(\Sigma))$ to this summand is non-canonical, and depends on a choice of splitting of the exact sequence 
		$$1\longrightarrow \mathrm{Im}\tau_1 \longrightarrow \mathcal{M} \longrightarrow \mathrm{Sp}_{2g}(\Z)\longrightarrow 1.$$
		which induces a splitting of the exact sequence
		$$1\longrightarrow (J_3/J_4)_{\Q} \longrightarrow (J_2/J_4)_{\Q} \longrightarrow (J_2/J_3)_{\Q} \longrightarrow 1,$$
		and thus an isomorphism $(J_2/J_4)_{\Q}\simeq (J_2/J_3)_{\Q} \times (J_3/J_4)_{\Q}.$ We remark that $J_2(\Sigma)/J_4(\Sigma)$ is abelian since $[J_2(\Sigma),J_2(\Sigma)]\subset J_4(\Sigma).$
		
		The map $\mathrm{Ab}(J_2(\Sigma))\longrightarrow [2^2]\oplus[31^2]$ is the so-called \textit{refined second Johnson homomorphism}, which was first introduced in \cite{Mor:survey} and \cite{MSS15}. This morphism vanishes on $J_4(\Sigma),$ as noted in the second equation of \cite{MSS20} between Equations (9) and (10), and as follows from our description. Finally, it can be noted that the map $\mathrm{Ab}(J_2(\Sigma))\longrightarrow [31^2]$ does not vanish on $[J_1(\Sigma),J_2(\Sigma)]$ (as we shall also see from our computations).
	\end{itemize}
\end{remark} 
\subsection{Irreducibility of some modular representations of $\Sp_{2g}(\Z)$}
\label{sec:irreducibility}
In this section, we gather some results about representations of $\Sp_{2g}(\Z)$ that are necessary for the proofs of our main theorems. We recall that we can view the group $\Sp_{2g}(\Z)$ as the image of the homology representation of a genus $g$ surface. To stay coherent with the other sections of the paper, we will write $H_1(\Sigma_g,\Z)$ for the fundamental representation of $\Sp_{2g}(\Z).$ 
Let also $\overline{\kappa}$ be the reduction mod $p$ of $\kappa.$

Since $\kappa(\omega\wedge h)=(g-1)h$ for any $h\in H_1(\Sigma,\Z),$ if $p$ divides $g-1$ then $\omega \wedge H_1(\Sigma,\F_p) \subset \Ker \overline{\kappa},$ and $\Ker \overline{\kappa} / (\omega \wedge H_1(\Sigma,\F_p))$ is a non-trivial subrepresentation of $\Lambda^3 H_1(\Sigma,\F_p) /\left( \omega \wedge H_1(\Sigma,\F_p)\right).$

If on the other hand $p$ does not divide $g-1,$ then $\Ker \overline{\kappa}$ is a direct summand of $\omega\wedge H_1(\Sigma,\F_p)$ in $\Lambda^3H_1(\Sigma,\F_p)$ and thus isomorphic to $V=\Lambda^3H_1(\Sigma,\F_p)/\left(\omega \wedge H_1(\Sigma,\F_p)\right).$
\begin{proposition}\label{prop:irreducibility_Lambda3}For any $g\geq 3$ and for any odd prime $p\geq 5,$ the representation $V=\Lambda^3 H_1(\Sigma,\F_p) /\left( \omega \wedge H_1(\Sigma,\F_p)\right)$ of $\Sp_{2g}(\Z)$ is irreducible if and only if $p$ does not divide $g-1.$

Moreover, if $p$ divides $g-1,$ then the only subrepresentations of $V$ are $\lbrace 0 \rbrace,$ $V$ and $\Ker \overline{\kappa} / (\omega \wedge H_1(\Sigma,\F_p)).$

\end{proposition}
We note that the above proposition is part of the results of \cite{PS83} on the composition factors of Weyl modules for $\Sp_{2g}(\Z)$ (see also \cite[Theorem 1.1]{Fou05}). However, for the convenience of the reader, we include an elementary proof.
\begin{proof}

Let $c_1,\ldots,c_{2g}=a_1,b_1,\ldots,a_g,b_g$ be a symplectic basis of $H_1(\Sigma,\F_p).$ We also write $c_i'=c_{i-1}$ if $i$ is even, and $c_i'=c_{i+1}$ if $i$ is odd, so that $\lbrace c_i,c_i'\rbrace$ is always a pair $\lbrace a_k,b_k\rbrace.$ We note that the vectors $c_i\wedge c_j\wedge c_k$ for $i<j<k$ and $(i,j,k)\neq (i,2g-1,2g)$ and $(2g-3,2g-2,k)$ form a basis of $V.$

In the whole proof, we write $K$ to mean $\Ker \overline{\kappa}\simeq V$ if $p$ does not divide $g-1,$ or $K=\Ker \overline{\kappa} / (\omega \wedge H_1(\Sigma,\F_p))$ if $p$ divides $g-1.$ 

\underline{Claim 1:} If $W$ contains a vector $c_i\wedge c_j\wedge c_k$ where $\omega(c_i,c_j)=\omega(c_j,c_k)=\omega(c_k,c_j)=0,$ then $W$ contains $K.$

Indeed, since $\Sp_{2g}(\Z)$ acts transitively on basis vectors $c_i\wedge c_j\wedge c_k$ with $c_i,c_j,c_k$ generating an isotropic subspace of $H_1(\Sigma,\F_p),$ the subspace $W$ would contain all such vectors. Since a generating set of $K$ consists of those vectors and the vectors $(a_i\wedge b_i)\wedge c_l-(a_j\wedge b_j)\wedge c_l$ where $c_l \notin \lbrace a_i,b_i,a_j,b_j\rbrace,$ it suffices to show that the latter vectors are also in $W.$ However,
\begin{multline*}t_{a_i+a_j}(b_i\wedge b_j\wedge c_l)=(b_i+a_i+a_j)\wedge (b_j+a_i+a_j)\wedge c_l
\\=b_i\wedge b_j\wedge c_l + b_i\wedge a_j\wedge c_l + a_i\wedge b_j \wedge c_l + \left( a_j\wedge b_j \wedge c_l -a_i\wedge b_i \wedge c_l \right).
\end{multline*}
Hence $W$ should also contain the vectors $a_j\wedge b_j \wedge c_l -a_i\wedge b_i \wedge c_l,$ and thus all of $K.$

\underline{Claim 2:} Assume now that $W$ contains a vector $w$ which has a non zero coefficient along a basis vector $c_i\wedge c_j \wedge c_k,$ where $\omega(c_i,c_j)=\omega(c_j,c_k)=\omega(c_k,c_i)=0.$ Then $W$ contains $K.$

Indeed, Claim 2 will follow from Claim 1 if we show that we can realize the projection $\pi$ from $V$ onto $\mathrm{Span}(c_i\wedge c_j \wedge c_k)$ as an element of $\F_p[\Sp_{2g}(\F_p)].$ For $\lambda$ a generator of $\F_p^*,$ let $\phi_{\lambda}^l\in \Sp_{2g}(\F_p)$ such that $\phi_{\lambda}^l(c_l)=\lambda c_l$ and $\phi_{\lambda}(c_l')=\lambda^{-1}c_l',$ and $\phi_{\lambda}(c_m)=c_m$ for the other basis vectors of $H_1(\Sigma,\F_p).$ Note that the endomorphism of $V$ induced by $\phi^l_{\lambda}$ (which we will abusively also denote by $\phi^l_{\lambda}$) acts diagonally in the basis $c_i\wedge c_j \wedge c_k$ of $V.$ Moreover, the eigenvalue is $\lambda$ if there is one $c_l$ and no $c_l'$ among $c_i,c_j,c_k,$ it is $1$ if there is either no $c_l$ and no $c_l'$ among $c_i,c_j,c_k$ or if $c_l$ and $c_l'$ both appear, and it is $\lambda^{-1}$ if only $c_l'$ appears. Then the endomorphism of $V$ defined by
$$\pi^l=-\underset{\lambda \in \F_p^*}{\sum} \lambda^{-1}\phi_{\lambda}^l$$
is the projection on the subspace spanned by basis vectors with one component equal to $c_l$ and no $c_l'$ component. The above follows from the classical identity that $\underset{\lambda \in \F_p^*}{\sum} \lambda^k=0$ if $p-1$ does not divide $k,$ and $=-1$ else.

Moreover, $\pi=\pi^i \pi^j \pi^k$ is the projection on $\mathrm{Span}(c_i\wedge c_j \wedge c_k).$

\underline{Claim 3:} $K$ is an irreducible representation.

Let $W\subset K$ be a non-trivial subrepresentation, let $v\neq 0 \in W.$ Let us assume that $v$ has no non-zero coefficient along any $c_i\wedge c_j\wedge c_k$ with $\omega(c_i,c_j)=\omega(c_j,c_k)=\omega(c_k,c_i)=0,$ since otherwise Claim 2 applies. Up to applying an element of $\Sp_{2g}(\Z)$ that permutes the basis $c_i$ of $H_1(\Sigma,\F_p),$ we may assume that $v$ has a non-zero coefficient along a vector basis of the form $a_i\wedge b_i \wedge a_g$ where $i<g.$ Then,  applying the projection $\pi^{2g-1}$ as in the proof of Claim 2 we can assume that:
$$v=\lambda_1 a_1\wedge b_1 \wedge a_g + \lambda_2 a_2\wedge b_2 \wedge a_g \ldots + \lambda_{g-1} a_{g-1} \wedge b_{g-1} \wedge a_g,$$
for some coefficients $\lambda_i \in \F_p.$ Since $v$ is not in $\omega \wedge H_1(\Sigma,\F_p),$ we can also assume that the $\lambda_i$ are not all equal. Without loss of generality, we assume that $\lambda_1\neq \lambda_2.$

Then we have 
$$t_{a_1+a_2}(v)-v=(\lambda_1-\lambda_2)a_1\wedge a_2\wedge a_g,$$
and therefore we can conclude by Claim 2.

\underline{Claim 4:} Assume that $p$ divides $g-1.$ Let $W$ be a subrepresentation of $V$ not included in $\Ker \overline{\kappa} / (\omega \wedge H_1(\Sigma,\F_p)).$ Then $W=V.$

Since $H_1(\Sigma,\F_p)$ is an irreducible representation, without loss of generality, we can assume that $W$ contains a vector $v$ such that $\overline{\kappa}(v) = a_g.$ We claim that we can assume that $v$ is of the form 
$$v=\lambda_1 a_1\wedge b_1\wedge a_g +\ldots + \lambda_{g-1} a_{g-1}\wedge b_{g-1}\wedge a_g.$$
Let us introduce maps
$$\tilde{\pi}^l= -\underset{\lambda \in \F_p^*}{\sum} \phi_{\lambda}^l$$
where the maps $\phi_{\lambda}^l$ were introduced in the proof of Claim 2. Then $\tilde{\pi}^l$ is the projection on the subspace spanned by basis vectors $c_i\wedge c_j\wedge c_k$ where either none of $c_i,c_j,c_k$ is $a_l$ or $b_l,$ or one is $a_l$ and another one is $b_l.$
Therefore we claim that $\tilde{\pi}^1 \ldots \tilde{\pi}^{g-1} \pi^g (v)$ is the form required and is in $W.$
We conclude similarly as in the proof of Claim 3: since $\overline{\kappa}(v)=a_g$ we have $\lambda_1+\ldots +\lambda_{g-1}=1,$ therefore they are not all equal. WLOG assume that $\lambda_1\neq \lambda_2,$ then $t_{a_1+a_2}(v)-v=(\lambda_1-\lambda_2)a_1\wedge a_2\wedge a_g,$ therefore $W$ contains $K$ by Claim 1, and thus contains all of $V.$

\end{proof}
\begin{lemma}\label{lemma:faithfulness_Lambda3}For any $g\geq 3$ and for any odd prime $p,$ the representation 
$$\Sp_{2g}(\F_p) \longrightarrow \mathrm{GL}\left(\Lambda^3 H_1(\Sigma,\F_p) /\left( \omega \wedge H_1(\Sigma,\F_p)\right)\right)$$ 
 is faithful.
\end{lemma}
\begin{proof}Note that for $\lambda \in \F_p^*,$ the map $\lambda \cdot \mathrm{id}$ acts as multiplication by $\lambda^2$ on $\omega,$ hence is in $\Sp_{2g}(\F_p)$ if and only if $\lambda=\pm 1.$ Moreover, $-\id$ acts by multiplication by $-1$ on $\Lambda^3 H_1(\Sigma,\F_p),$ and therefore $-\id$ is not in the kernel of this representation.

Now, it is well known that $\mathrm{PSp}_{2g}(\F_p)$ is a simple group for any $g\geq 2$ and any prime $p\geq 3,$ and that the only normal subgroups of $\Sp_{2g}(\F_p)$ are the trivial subgroup, $\Sp_{2g}(\F_p)$ and $Z(\Sp_{2g}(\F_p))=\lbrace \pm \id \rbrace,$ see for example \cite{Dieu}. Since the kernel does not contain $-\id,$ it must be trivial. 

\end{proof}
\begin{lemma}\label{lemma:abs_irred}Let $G$ be a group and $V$ a free $\Z$-module of finite rank, and let $\rho: G\longrightarrow \mathrm{Aut}(V)$ be a representation. Assume that the induced representation $\overline{\rho}:G\longrightarrow \mathrm{Aut}(V \underset{\Z}{\otimes} \overline{\Q})$ is irreducible. Then for all large enough $p,$ the representation $\rho_p:G \longrightarrow \mathrm{Aut}(V \underset{\Z}{\otimes} \F_p)$ is irreducible.
\end{lemma}
\begin{proof}
Since $\overline{\Q}$ is algebraically closed, $\rho$ is absolutely irreducible if and only if $\overline{\Q}[\rho(G)]=\mathrm{End}_{\overline{\Q}}(V \underset{\Z}{\otimes} \overline{\Q}).$ However, since $\rho(G) \subset \mathrm{End}_{\Q}(V \underset{\Z}{\otimes} \Q),$ this is equivalent to $\Q[\rho(G)]=\mathrm{End}_{\Q}(V\underset{\Z}{\otimes} \Q).$ Since $V$ has finite rank, we can conclude that there is an integer $D$ such that $D\cdot \mathrm{End}(V) \subset \Z[\rho(G)].$ Now take $p$ be any prime number not dividing $D,$ reducing mod $p$ we get that $\mathrm{End}(V \underset{\Z}{\otimes} \F_p) \subset \F_p[\rho_p(G)],$ which implies that $\rho_p$ is irreducible. 
\end{proof}
\subsection{Simply intersecting pairs and the contraction map}
\label{sec:SIP}
Proposition \ref{prop:irreducibility_Lambda3} shows the importance of understanding the kernel of the contraction map 
$$\widetilde{\kappa}: \Lambda^3 H_1(\Sigma,\Z)/(\omega \wedge H_1(\Sigma,\Z)) \rightarrow H_1(\Sigma,\Z/(g-1)\Z).$$ In this section, we will introduce some elements of the Torelli subgroup $J_1(\Sigma)$ whose images are in this kernel.

\begin{definition} \label{def:SIP}Let $\alpha$ and $\beta$ be two non separating curves on $\Sigma,$ such that $\alpha$ and $\beta$ are geometrically intersecting twice, and their algebraic intersection is zero. Then we call the pair of curves $(\alpha,\beta)$ a simply intersecting pair (or SIP) and the element $[t_{\alpha},t_{\beta}]$ a SIP-map.
\end{definition}

The work of Childers \cite{Chi12} shows that SIP-maps are in the kernel of the contraction map. More precisely, if $[t_{\alpha},t_{\beta}]$ is a SIP-map then a regular neighborhood of $\alpha \cup \beta$ is a four-holed sphere with some curves $x,y,z,w$ as boundary components. Then \cite[Main Result 2]{Chi12} states that 
$$\tau_1([t_{\alpha},t_{\beta}])=[x]\wedge [y]\wedge [z],$$
where $\tau_1: J_1(\Sigma) \rightarrow \Lambda^3 H_1(\Sigma,\Z)/(\omega \wedge H_1(\Sigma,\Z))$ is the first Johnson homomorphism.

From this we get the following:
\begin{lemma}\label{lemma:SIP} There exists a SIP-map $[t_{\alpha},t_{\beta}]$ in a one-holed genus $3$ surface such that $\tau_1([t_{\alpha},t_{\beta}])$ is a primitive element of $\ker(\widetilde{\kappa}) \subset \Lambda^3 H_1(\Sigma,\Z)/\left( \omega \wedge H_1(\Sigma,\Z)\right).$ 
\end{lemma}
\begin{proof}
We will construct a simply intersecting pair $(\alpha,\beta)$ in $\Sigma$ such that $x,y,z,w$ are the boundary components of $N(\alpha\cup \beta).$
Thanks to Childers' formula, we see that the conclusion of the lemma will hold if the curves $x,y,z$ are non separating  in $\Sigma$ and if their homology classes can be completed into a basis of $H_1(\Sigma,\Z).$ The latter part is true if the union $x \cup y \cup z$ is also non-separating.  
Start with a $4$-holed sphere with boundary components $x,y,z,w.$ Gluing tubes or pants to connect the curves $x,y,z,w,$ it is easy to see that one can construct a one-holed genus $3$ surface where this holds. 
\end{proof}

\subsection{Casson invariant and quantum invariant}
For $M$ an integral homology $3$-sphere, let $\lambda(M) \in \Z$ be its Casson invariant. An important theorem by Murakami relates quantum invariants with Casson invariants. Recall that for $p \ge 5$ a prime, $\zeta_p$ denotes a $p$-th primitive root of unity and $h=1-\zeta_p$.

\begin{theorem}[\cite{Mu95}] \label{thm:Oht}
If $M$ is an integral homology $3$-sphere, then $$Z_p(M) = Z_p(S^3) \big( 1+6h \lambda(M) \big) \quad \mathrm{mod} \ h^2$$

\end{theorem}

Now let $\Sigma$ be a surface of genus at least $6$, we choose $H_1$ and $H_2$ two handlebodies with boundary $\Sigma$ such that $H_1 \underset{\mathrm{Id}}{\cup} \bar{H_2} = S^3$. If $\varphi \in J_1(\Sigma)$, the $3$-manifold $H_1 \underset{\varphi}{\cup} \bar{H_2}$ is an integral homology $3$-sphere. In \cite{M89}, Morita showed that the map $$\lambda_{H_1,H_2} : \varphi \in J_2(\Sigma) \mapsto \lambda\big(H_1 \underset{\varphi}{\cup} \bar{H_2}\big) \in \Z$$ is a group homomorphism. Moreover he proved that when restricted to $J_3(\Sigma)$, the map $\lambda_{H_1,H_2}$ does not depend on the choice of the handlebodies $H_1$ and $H_2$. For $\varphi \in J_3(\Sigma)$, we will simply write $\lambda(\varphi)$ for $\lambda_{H_1,H_2}(\varphi)$. The map $\lambda : \varphi \in J_3(\Sigma) \mapsto \lambda(\varphi) \in \Z$ is invariant by conjugation under $\Mod(\Sigma)$ and is proportional to the so-called core of the Casson invariant. We will also need the following theorem by Hain
\begin{theorem}[\cite{H97}] \label{thm:Hain}
For $k \ge 3$, $\lambda(J_k) \neq \{ 0 \}$.
\end{theorem}
Although we will not need it here, we note that for the case $k=4$, Faes in \cite{F22} showed that $\lambda(J_4) = \Z$.
\section{$h$-adic expansion of quantum representations and Johnson filtration}
\label{sec:proofs}
\subsection{$h$-adic expansion of quantum representations}
In all this section, we fix a prime $p\geq 5.$ As in the previous section, let $h=1-\zeta_p\in \Z[\zeta_p].$ We recall that $h$ is a prime in $\Z[\zeta_p],$ and that $p$ is equal to $xh^{p-1}$ for some unit $x\in \Z[\zeta_p].$ We denote by $\rho_{p,k}$ the representation $\rho_p$ reduced modulo $h^k,$ whose coefficients then belong to $\Z[\zeta_p]/(h^k).$
\begin{lemma}\label{lemma:kernel_modhk} Let $N_k \triangleleft \Mod(\Sigma)$ be the kernel of $\rho_{p,k}.$ Then:
\begin{itemize}
\item[(i)]$\forall k,l \geq 1,$ one has $[N_k,N_l] \subset N_{k+l}.$
\item[(ii)]$\forall k \geq 1,$ one has $N_k^{p} \subset N_{k+p-1},$ where $N_k^p$ is the subgroup of $\Mod(\Sigma)$ generated by $p$-th powers of elements of $N_k.$
\end{itemize}
\end{lemma}
\begin{proof}
We warn the reader that in the following proof, we consider all equalities to hold up to a power of $\zeta_p.$
An element $f\in \Mod (\Sigma)$ is in $N_k$ if and only if $\rho_{p}(f)=\id_{\mathcal{S}_p(\Sigma)} + h^k u,$ where $u\in \mathrm{End}_{\Z[\zeta_p]}(\mathcal{S}_p(\Sigma)).$ Let $f\in N_k$ and $g\in N_l$ and write $\rho_p(f)=id_{\mathcal{S}_p(\Sigma)} + hu$ and $\rho_p(g)=\id +h^lv.$ Then
$$\rho(f)^{-1}=\id_{\mathcal{S}_p(\Sigma)} -h^ku +h^{2k}u^2 -\ldots \ \mathrm{mod} \ h^{k+l}.$$
One can write the same formula for $\rho(g)^{-1}$ modulo $h^{k+l}.$ Then a direct computation shows that $\rho([f,g])=id_{\mathcal{S}_p(\Sigma)}$ modulo $h^{k+l}.$

As for point (ii), we recall that $p$ is equal to $h^{p-1}$ up to a unit in $\Z[\zeta_p].$ Let us assume again $f\in N_k$ and $\rho_p(f)=id +hu,$ then we have
$$\rho_p(f)^p=id +\binom{p}{1} h^ku + \binom{p}{2} h^{2k}u^2 + \ldots + h^{pk}u^{p}$$
Note that $p$ divides all binomial coefficients $\binom{p}{j}$ with $1\leq j \leq p-1,$ so all terms $h^{jk}\binom{p}{j} u^j$ with $1\leq j \leq p-1$ are zero modulo $h^{k+p-1}.$ Since $pk\geq k+p-1,$ the last term is also zero modulo $h^{k+p-1}$ and we get the second claim. 
\end{proof}
\subsection{Analysis of $\mathrm{Ker}(\rho_{p,1})$}
\label{sec:rho_p_modh}
In this section, we will study the representation $\rho_{p,1}$ restricted to the Torelli group $J_1(\Sigma)$ and give a proof of Theorem \ref{thm:main-thm1}. First, note that since we restrict $\rho_{p,1}$ to $J_1(\Sigma),$ we can consider it a linear representation instead of just a projective representation. For $G$ a group, we write $\mathrm{Ab}(G)$ for its abelianization, and $\mathrm{Ab}_p(G)$ for its mod $p$ abelianization: $\mathrm{Ab}_p(G)=G/\langle [G,G],G^p\rangle.$
\begin{lemma}\label{lemma:modh_abelian}The morphism 
$$\rho_{p,1}: J_1(\Sigma) \longrightarrow \mathrm{Aut}(\mathcal{S}_p(\Sigma)) \ \mathrm{mod} \ h$$ factors through $\mathrm{Ab}_p(J_1(\Sigma)).$
\end{lemma}
\begin{proof}
It is a consequence of Corollary \ref{cor:decompJohnson} that $J_2(\Sigma) \subset \mathrm{Ker}(\rho_{p,1}).$ Since $J_1(\Sigma)/J_2(\Sigma)$ is abelian, $\rho_{p,1}$ factors through $\mathrm{Ab}(J_1(\Sigma)).$ Moreover, $J_1(\Sigma)$ is generated by bounding pairs by Johnson's theorem \cite{Joh1}. Therefore it suffices to show that $p$-th powers of bounding pairs are in $\mathrm{Ker}(\rho_{p,1}).$ However, the $p$-th power of a bounding pair $\tau_c \tau_{c'}^{-1}$ is $\tau_c^p \tau_{c'}^{-p},$ a product of $p$-powers of Dehn twists, which is in $\mathrm{Ker}(\rho_p).$
\end{proof}
We can put a $\Mod(\Sigma)$-module structure on $\rho_{p,1}(J_1(\Sigma))$ by defining for $f\in \Mod(\Sigma), g \in J_1(\Sigma):$

$$f\cdot \rho_{p,1}(g) = \rho_{p,1}(fgf^{-1}).$$
However, since $\rho_{p,1}$ is abelian  on $J_1(\Sigma),$ this induces a $\Mod(\Sigma)/J_1(\Sigma)\simeq \Sp_{2g}(\Z)$-representation structure on $ \rho_{p,1}(J_1(\Sigma)).$ 
\begin{proposition}\label{prop:image_modh}For any $g\geq 3,$ and any prime $p\geq 5$ the representation $\rho_{p,1}$ induces an isomorphism of $\Sp_{2g}(\Z)$ representations $Ab_p(J_1(\Sigma)) \simeq \rho_{p,1}(J_1(\Sigma)).$
\end{proposition}
\begin{proof} We will first consider the case where $p$ does not divide $g-1.$ Thanks to Lemma \ref{lemma:modh_abelian}, the map $\rho_{p,1}:J_1(\Sigma) \longrightarrow \rho_{p,1}(J_1(\Sigma))$ induces a surjective morphism of $\Sp_{2g}(\Z)$-representations from $\mathrm{Ab}_p(J_1(\Sigma))$ to $ \rho_{p,1}(J_1(\Sigma)).$ However, by Proposition \ref{prop:irreducibility_Lambda3} and Theorem \ref{thm:abelianization_Torelli}, when $p\geq 5$ is not a divisor of $g-1,$ we have that $Ab_p(J_1(\Sigma))\simeq \Lambda^3 H_1(\Sigma,\F_p)/ \left(\omega \wedge H_1(\Sigma,\F_p)\right)$ is an irreducible representation of $\Sp_{2g}(\Z).$ This means that the morphism induced by $\rho_{p,1}$ is either an isomorphism or the trivial morphism. 
However, since by Lemma \ref{lemma:bounding_pair}, the image of a bounding pair of genus $1$ by $\rho_{p,1}$ is not trivial, we have that $\rho_{p,1}$ is not trivial, and therefore induces an isomorphism $Ab_p(J_1(\Sigma)) \simeq  \rho_{p,1}(J_1(\Sigma)).$

Now, let us treat the case where $p$ divides $g-1.$ We still have that $\rho_{p,1}$ induces a surjective morphism of $\Sp_{2g}(\Z)$ representation $Ab_p(J_1(\Sigma)) \longrightarrow \rho_{p,1}(J_1(\Sigma)),$ however since $Ab_p(J_1(\Sigma)) \simeq \Lambda^3 H_1(\Sigma,\F_p) /(\omega \wedge H_1(\Sigma,\F_p)),$ the kernel might not be trivial. Since the kernel of this map is a $\Sp_{2g}(\Z)$-subrepresentation, by Proposition \ref{prop:irreducibility_Lambda3}, we only have to exclude the possibility that $\Ker \rho_{p,1} = \Ker \overline{\kappa}/\left(\omega \wedge H_1(\Sigma,\F_p)\right).$ 

For this, let us introduce $S,$ a subsurface of $\Sigma$ of genus $g'=3$ with one boundary component, which is a separating curve in $\Sigma.$ By Lemma \ref{lemma:SIP}, there is a SIP-map, $[t_a,t_b],$ with support on $S$, and such that $\tau_1([t_a,t_b])$ is a nonzero element of $\Ker \overline{\kappa}$ for any prime $p\geq 5.$ However, since $p$ does not divide $g'-1=2$, we have that the mapping class induced by $[t_a,t_b]$ on $\hat{S}$ is not in $\Ker \rho_{p,1}.$ By Lemma \ref{lemma:support}, since the boundary of $S$ is separating in $\Sigma,$ the mapping class $[t_a,t_b]$ as an element of $\Mod(\Sigma),$ is also not in $\Ker \rho_{p,1}.$ Hence $\Ker \rho_{p,1}$ does not contain $\Ker \overline{\kappa},$ and induces an isomorphism $Ab_p(J_1(\Sigma)) \simeq \rho_{p,1}(J_1(\Sigma)).$  

\end{proof}
\begin{proof}[Proof of Theorem \ref{thm:main-thm1}]First, let us note that the inclusion $[J_1(\Sigma),J_1(\Sigma)]T_p\subset \Ker \rho_{p,1}$ is clear, since $T_p\subset \Ker \rho_p$ and $[J_1(\Sigma),J_1(\Sigma)]\subset J_2(\Sigma)\subset \Ker \rho_{p,1}$ as noted before. Let $f \in \Ker \rho_{p,1}.$ Then $f$ acts trivially on $\rho_{p,1}(J_1(\Sigma))$ by conjugation. By Proposition \ref{prop:image_modh}, we have that $\rho_{p,1}(J_1(\Sigma)) \simeq \Lambda^3 H_1(\Sigma,\F_p)/ \left(\omega \wedge H_1(\Sigma,\F_p)\right),$ and by Lemma \ref{lemma:faithfulness_Lambda3}, the group $\Sp_{2g}(\F_p)$ acts faithfully on $\Lambda^3 H_1(\Sigma,\F_p)/ \left(\omega \wedge H_1(\Sigma,\F_p)\right).$ Therefore the image of $f$ in $\Sp_{2g}(\F_p)$ is trivial, that is, $f$ is in the Torelli mod $p$ subgroup. Since $g\geq 2,$ the Torelli mod $p$ subgroup is the subgroup $J_1(\Sigma)T_p$ (this is a consequence of the classical fact that the $p$-congruence subgroup of  $\Sp_{2g}(\Z)$ is generated by $p$-th powers of transvections).
Now since the subgroup $T_p$ generated by $p$-th powers of Dehn twists is contained in $\Ker \rho_p,$ we must have $f=f_1 f_2$ with $f_2\in T_p$ and $f_1 \in J_1(\Sigma) \cap \Ker \rho_{p,1}=\Ker \rho_{p,1}|_{J_1(\Sigma)}.$ However, by Proposition \ref{prop:image_modh}, this kernel is the same of the kernel of the mod $p$ abelianization of the Torelli group. Since the Torelli group is generated by bounding pairs, whose $p$-th powers are in $T_p,$ this is the same as the subgroup generated by $[J_1(\Sigma),J_1(\Sigma)]$ and $p$-powers of bounding pairs, and therefore $f\in [J_1(\Sigma),J_1(\Sigma)]T_p.$
\end{proof}
\subsection{Analysis of $\mathrm{Ker}(\rho_{p,2})$}
\label{sec:rho_p_modh2}
In this section, we will analyze the kernel of $\rho_{p,2}$ restricted to $J_2(\Sigma),$ and deduce Theorem \ref{thm:main-thm2}.
\begin{lemma}\label{lemma:modh2_abelian} The morphism $\rho_{p,2}:J_2(\Sigma) \longrightarrow \rho_{p,2}(J_2(\Sigma))$ factors through $Ab_p(J_2(\Sigma)).$ 
\end{lemma}
\begin{proof}
We know that $J_2(\Sigma) \subset N_1 = \Ker \rho_{p,1}$ by Corollary \ref{cor:decompJohnson}. By Lemma \ref{lemma:kernel_modhk}, we have that $[N_1,N_1] \subset N_2=\Ker \rho_{p,2},$ hence $\rho_{p,2}|_{J_2(\Sigma)}$ is abelian. Moreover, since $J_2(\Sigma)$ is the Johnson subgroup which is well-known to be generated by separating twists, whose $p$-th powers  are in $\Ker \rho_p,$ the morphism $\rho_{p,2}|_{J_2(\Sigma)}$ actually factors through $Ab_p(J_2(\Sigma)).$
\end{proof}
Now, as in the previous section, we put a $\Mod(\Sigma)$-module structure on $ \rho_{p,2}(J_2(\Sigma))$ by defining for $f\in \Mod(\Sigma), g \in J_2(\Sigma):$
$$f \cdot \rho_{p,2}(g) = \rho_{p,2}(fgf^{-1}).$$
Again, since $\rho_{p,2}|_{J_2(\Sigma)}$ is abelian, this actually induces a $\mathcal{M}=\Mod(\Sigma)/J_2(\Sigma)$-module structure. Note that $\mathcal{M}\simeq \left(\Lambda^3 H_1(\Sigma,\Z)/\left( \omega \wedge H_1(\Sigma,\Z) \right)\right) \rtimes \Sp_{2g}(\Z),$ hence we get also a $\Sp_{2g}(\Z)$-module structure on $\rho_{p,2}(J_2(\Sigma)).$
\begin{proposition}\label{prop:image_modh2}Let $\Sigma$ be a closed surface of genus $g\geq 6.$ For all large enough primes $p,$ the map 
$$\rho_{p,2}:Ab_p(J_2(\Sigma)) \longrightarrow \rho_{p,2}(J_2(\Sigma))$$ is an isomorphism of $\Mod(\Sigma)/J_2(\Sigma)$ modules. 
\end{proposition}
\begin{proof}
By a theorem of Church, Ershov and Putman \cite{CEP},  if $g\geq 6$ then $J_2(\Sigma)$ is finitely generated. This implies that $Ab(J_2(\Sigma))$ is an abelian group of finite rank, and therefore that $Ab(J_2(\Sigma))$ has no $p$-torsion for any large enough $p.$ If that is the case then $Ab_p(J_2(\Sigma))$ is isomorphic to $Ab^f(J_2(\Sigma)) \underset{\Z}{\otimes} \F_p,$ where $Ab^f(J_2(\Sigma))$ denotes the free part of the abelianization of $J_2(\Sigma).$

By Theorem \ref{thm:abelianization_Johnson}, the rational abelianization $\Ab_{\Q}(J_2(\Sigma))$ of $J_2(\Sigma)$ is a sum of $3$ absolutely irreducible $\Sp_{2g}(\Z)$ representations. Moreover $\Ab_{\Q}(J_2(\Sigma))\simeq \Ab^f(J_2(\Sigma))\underset{\Z}{\otimes} \Q.$ By Lemma \ref{lemma:abs_irred}, we get that $\Ab_p(J_2(\Sigma))$ is also a sum of three irreducible representations of $\Sp_{2g}(\Z)$ over $\F_p,$ whenever $p$ is large enough. Then $\rho_{p,2}(J_2(\Sigma))$ is a quotient $\Sp_{2g}(\Z)$-representation of $\Ab_p(J_2(\Sigma)).$ Note that $\rho_{p,2}(J_2(\Sigma))$ is also a quotient $\mathcal{M}$-representation of $\Ab_p(J_2(\Sigma)),$ where $\mathcal{M}=\mathrm{Mod}(\Sigma)/J_2(\Sigma).$ 

We have to show that the $3$ irreducible $\Sp_{2g}(\Z)$-summands in $\Ab_p(J_2(\Sigma))$ survive in $\rho_{p,2}(J_2(\Sigma)).$ 

First, by Theorem \ref{thm:Hain}, we can find $\varphi \in J_4(\Sigma)$ with $\lambda(\varphi) \neq 0$. In particular by Remark \ref{rk:invarianceOfFactors}, $\varphi$ projects to a non-zero element in the $\Q$-summand of $\Ab_{\Q}(J_2(\Sigma))$. Hence for $p$ big enough $\lambda(\varphi) \neq 0$ in $\mathbb{F}_p$ and $\varphi$ projects to a non-zero element of the $\mathbb{F}_p$-summand of $\Ab_p(J_2(\Sigma))$. Now let $H_1$ and $H_2$ be two handlebodies with boundary $\Sigma$ such that $H_1 \underset{\mathrm{Id}}{\cup} \bar{H_2} = S^3$. By Theorem \ref{thm:Oht}, 
$$Z_p(H_1 \underset{\varphi}{\cup} \bar{H_2}) =Z_p(S^3) \big( 1+6h \lambda(\varphi) \big)  \quad \mathrm{mod} \ h^2,$$ which implies that \begin{equation} \label{eq:scal} Z_p(H_1 \underset{\varphi}{\cup} \bar{H_2}) \neq Z_p(S^3) \quad \mathrm{mod} \ h^2 \end{equation}
 On the other hand by the TQFT axioms $Z_p(H_1 \underset{\varphi}{\cup} \bar{H_2}) = \langle Z_p(H_1),\rho_p(\varphi) Z_p(H_2) \rangle$ and $Z_p(S^3) = \langle Z_p(H_1), Z_p(H_2) \rangle$. As the vectors $Z_p(H_1)$ and $Z_p(H_2)$ belong to the lattice $\mathcal{S}_p(\Sigma)$, Equation (\ref{eq:scal}) implies that $\rho_{p,2}(\varphi)$ is not trivial. Hence the $\mathbb{F}_p$-summand of $\Ab_p(J_2(\Sigma))$ survives in $\rho_{p,2}(J_2(\Sigma)).$

 Next, by Lemma \ref{lemma:not_J1_invariant}, $\rho_{p,2}(J_2(\Sigma))$ is not invariant under the $J_1(\Sigma)$-action. However, the first and second summands of $\Ab_p(J_2(\Sigma))$ are both invariant under the $J_1(\Sigma)$-action, hence the last summand $[31^2]$ has to survive in $\rho_{p,2}(J_2(\Sigma)).$ 

Finally, assume that the factor $[2^2]$ is killed in $\rho_{p,2}(J_2(\Sigma)).$ Then, we claim that $\rho_{p,2}([J_1(\Sigma),J_2(\Sigma)])$ would have codimension $1$ in $\rho_{p,2}(J_2(\Sigma)).$ Indeed, $\Ab_p([J_1(\Sigma),J_2(\Sigma)])$ is a sub $Sp_{2g}(\Z)$-representation of $Ab_p(J_2(\Sigma))$ since $[J_1(\Sigma),J_2(\Sigma)]$ is stable under $\mathrm{Mod}(\Sigma)$-conjugation, and we have shown that it is non trivial  (since it is not killed by $\rho_{p,2}$). However, by Remark \ref{rk:invarianceOfFactors}, $Ab_p([J_1(\Sigma),J_2(\Sigma)])$ lies entirely in $[31^2].$ 

Now, we appeal to Lemma \ref{lemma:independence}. The images of the separating Dehn twists of the lollipop tree decomposition of $\Sigma$ span a $\F_p$ subspace of $\rho_{p,2}(J_2(\Sigma))$ of dimension $2g-3,$ which has trivial intersection with $\rho_{p,2}([J_1(\Sigma),J_2(\Sigma)]),$ since by Corollary \ref{cor:decompJohnson}-(iii), $\rho_{p,2}([J_1(\Sigma),J_2(\Sigma)])$ consists of matrices whose diagonal is $1 \ (\mathrm{mod} \ h^2).$ Therefore, $\rho_{p,2}([J_1(\Sigma),J_2(\Sigma)])$ has codimension larger than $1$ in $\rho_{p,2}(J_2(\Sigma)),$ and hence the factor $[2^2]$ must survive.  
\end{proof}
\begin{proof}[Proof of Theorem \ref{thm:main-thm2}]
	The inclusion $[J_2(\Sigma),J_2(\Sigma)]T_p \subset \ker \rho_{p,2}$ is clear, since $T_p\subset \ker \rho_p,$ and $\rho_{p,2}$ restricted to $J_2(\Sigma)$ has abelian image by Lemma \ref{lemma:modh2_abelian}.
	
By Theorem \ref{thm:main-thm1}, we have that $\Ker \rho_{p,1} \subset [J_1(\Sigma),J_1(\Sigma)]T_p \subset J_2(\Sigma)T_p.$ Since $T_p\subset \Ker \rho_p,$ we only need to describe which elements of $J_2(\Sigma)$ are in $\Ker \rho_{p,2}.$ However, by Proposition \ref{prop:image_modh2}, for $p$ large enough, we have that $f\in J_2(\Sigma)$ is in $\Ker \rho_{p,2}$ if and only if $f$ is in the kernel of the mod $p$ abelianization morphism. Since $J_2(\Sigma)$ is generated by separating twists, the kernel of the mod $p$ abelianization is generated by commutators in $J_2(\Sigma)$ and $p$-th powers of separating Dehn twists, both of which are in $[J_2(\Sigma),J_2(\Sigma)]T_p.$

\end{proof}
\section{Further comments}
A consequence of Corollary \ref{cor:decompJohnson} is that it allows us to define a morphism from $\rho_{p,2}(J_2(\Sigma))$ to $\F_p:$
\begin{lemma}\label{lemma:coeurCasson} For $f\in J_2(\Sigma),$ let us write $\rho_{p,2}(f)=id_{S_p(\Sigma)}+ h\begin{pmatrix}
 A_1(f) & A_3(f) \\ 0 & A_2(f)
\end{pmatrix}$ as in Corollary \ref{cor:decompJohnson}. Then the map
$$ \begin{array}{rccl}
d': & J_2(\Sigma) & \longrightarrow & \F_p
\\  & f & \longrightarrow & \mathrm{Tr} (A_1(f)) \ \mathrm{mod} \ h\end{array} $$
 is a $\mathrm{Mod}(\Sigma)$-invariant morphism.

\end{lemma}
\begin{proof}
The fact that $d'$ is a morphism is a direct consequence of Corollary \ref{cor:decompJohnson}-(ii), while the fact that it is invariant under $\mathrm{Mod}(\Sigma)$-action is a direct consequence of Theorem \ref{thm:decomp}.
\end{proof}
Lemma \ref{lemma:coeurCasson} makes it tempting to think of the morphism $d'$ as (a scalar multiple of) the reduction mod $p$ of the core of the Casson invariant, which is the unique (up to scalar) $\mathrm{Mod}(\Sigma)$-invariant morphism $J_2(\Sigma)\rightarrow \F_p$ by the work of Morita \cite{M89}. Unfortunately, nothing garanties that the morphism  $d'$ is not identically zero. 

To check whether this is the case, it is possible to compute $d'$ on a separating Dehn twist of genus $1,$ thanks to Gilmer and Masbaum's formulae for the dimensions of the odd-even decomposition of $\mathcal{S}_p(\Sigma)$ (see \cite[Proposition 7.2]{GM14}). By performing this analysis, it is possible to check that for each genus $g,$ the morphism $d'$ is non trivial for small values of $p,$ and trivial for all large enough primes $p.$ Since the proof of this negative result is a bit involved, we will not include it here.

As another remark, a naive approach would be to define another $\mathrm{Mod}(\Sigma)$-invariant morphism $d''$ by setting $d''(f)= \mathrm{Tr}(A_1(f)+A_2(f)) \ \mathrm{mod} \ h.$ Then as a consequence of Verlinde's formula, it would be possible to show that $d''$ is always the trivial morphism, for any genus $g$ and any prime $p\geq 5.$
\bibliographystyle{hamsalpha}
\bibliography{biblio}

\providecommand{\bysame}{\leavevmode\hbox to3em{\hrulefill}\thinspace}
\providecommand{\href}[2]{#2}
\providecommand{\eprint}{\begingroup \urlstyle{rm}\Url}
\begin{thebibliography}{BHMV92}

\bibitem[And06]{A06}
J\o rgen~Ellegaard Andersen, \emph{Asymptotic faithfulness of the quantum
  {${\rm SU}(n)$} representations of the mapping class groups}, Ann. of Math.
  (2) \textbf{163} (2006), no.~1, 347--368.

\bibitem[BHMV92]{BHMV}
C.~Blanchet, N.~Habegger, G.~Masbaum, and P.~Vogel, \emph{{Three-manifold
  invariants derived from the {K}auffman bracket}}, {Topology} \textbf{31}
  (1992), no.~4, {685--699}.

\bibitem[CEP22]{CEP}
Thomas Church, Mikhail Ershov, and Andrew Putman, \emph{On finite generation of
  the {J}ohnson filtrations}, J. Eur. Math. Soc. (JEMS) \textbf{24} (2022),
  no.~8, 2875--2914.

\bibitem[Chi12]{Chi12}
Leah~R. Childers, \emph{Simply intersecting pair maps in the mapping class
  group}, J. Knot Theory Ramifications \textbf{21} (2012), no.~11, 1250107, 21.

\bibitem[DHP14]{DHP14}
Alexandru Dimca, Richard Hain, and Stefan Papadima, \emph{The abelianization of
  the {J}ohnson kernel}, J. Eur. Math. Soc. (JEMS) \textbf{16} (2014), no.~4,
  805--822.

\bibitem[Die71]{Dieu}
Jean~A. Dieudonn\'{e}, \emph{La g\'{e}om\'{e}trie des groupes classiques},
  Ergebnisse der Mathematik und ihrer Grenzgebiete, Band 5, Springer-Verlag,
  Berlin-New York, 1971, Troisi\`eme \'{e}dition.

\bibitem[DM22]{DM22}
Bertrand Deroin and Julien Marché, \emph{Toledo invariants of topological
  quantum field theories}, 2022.

\bibitem[DP13]{DP13}
Alexandru Dimca and Stefan Papadima, \emph{Arithmetic group symmetry and
  finiteness properties of {T}orelli groups}, Ann. of Math. (2) \textbf{177}
  (2013), no.~2, 395--423.

\bibitem[Fae22]{F22}
Quentin Faes, \emph{Triviality of the {{\(J_4\)}}-equivalence among homology
  3-spheres}, Trans. Am. Math. Soc. \textbf{375} (2022), no.~9, 6597--6620
  (English).

\bibitem[Fou05]{Fou05}
Sebastien Foulle, \emph{Characters of the irreducible representations with
  fundamental highest weight for the symplectic group in characteristic p},
  2005, \eprint{math/0512312}.

\bibitem[Fun99]{F99}
Louis Funar, \emph{On the {TQFT} representations of the mapping class groups},
  Pacific J. Math. \textbf{188} (1999), no.~2, 251--274.

\bibitem[FWW02]{FWW02}
Michael~H. Freedman, Kevin Walker, and Zhenghan Wang, \emph{Quantum {$\rm
  SU(2)$} faithfully detects mapping class groups modulo center}, Geom. Topol.
  \textbf{6} (2002), 523--539.

\bibitem[GM07]{GM07}
Patrick~M. Gilmer and Gregor Masbaum, \emph{Integral lattices in {TQFT}}, Ann.
  Sci. \'{E}cole Norm. Sup. (4) \textbf{40} (2007), no.~5, 815--844.

\bibitem[GM14]{GM14}
\bysame, \emph{Irreducible factors of modular representations of mapping class
  groups arising in integral {TQFT}}, Quantum Topol. \textbf{5} (2014), no.~2,
  225--258.

\bibitem[GM17]{GM17}
\bysame, \emph{An application of {TQFT} to modular representation theory},
  Invent. Math. \textbf{210} (2017), no.~2, 501--530.

\bibitem[Hai97]{H97}
Richard Hain, \emph{Infinitesimal presentations of the torelli groups}, Journal
  of the American Mathematical Society \textbf{10} (1997), no.~3, 597--651.

\bibitem[Joh83]{Joh1}
Dennis Johnson, \emph{The structure of the {T}orelli group. {I}. {A} finite set
  of generators for {${\mathcal{I}}$}}, Ann. of Math. (2) \textbf{118} (1983),
  no.~3, 423--442.

\bibitem[Joh85]{Joh3}
\bysame, \emph{The structure of the {T}orelli group. {III}. {T}he
  abelianization of {$\mathcal{T}$}}, Topology \textbf{24} (1985), no.~2,
  127--144.

\bibitem[Kor19]{Kor19}
Julien Korinman, \emph{Decomposition of some {W}itten-{R}eshetikhin-{T}uraev
  representations into irreducible factors}, SIGMA Symmetry Integrability Geom.
  Methods Appl. \textbf{15} (2019), Paper No. 011, 25.

\bibitem[KS16]{KS16}
Thomas Koberda and Ramanujan Santharoubane, \emph{Quotients of surface groups
  and homology of finite covers via quantum representations}, Invent. Math.
  \textbf{206} (2016), no.~2, 269--292.

\bibitem[Mas99]{M99}
Gregor Masbaum, \emph{An element of infinite order in {TQFT}-representations of
  mapping class groups}, Low-dimensional topology ({F}unchal, 1998), Contemp.
  Math., vol. 233, Amer. Math. Soc., Providence, RI, 1999, pp.~137--139.

\bibitem[Mas08]{M08}
\bysame, \emph{On representations of mapping class groups in integral tqft. in:
  Invariants in low-dimensional topology}, Oberwolfach {Rep}. 5, {No}. 2,
  1202-1205 (2008)., 2008.

\bibitem[Mor89]{M89}
Shigeyuki Morita, \emph{Casson's invariant for homology 3-spheres and
  characteristic classes of surface bundles i}, Topology \textbf{28} (1989),
  no.~3, 305--323.

\bibitem[Mor91]{Mor}
Shigeyuki Morita, \emph{On the structure of the {T}orelli group and the
  {C}asson invariant}, Topology \textbf{30} (1991), no.~4, 603--621.

\bibitem[Mor99]{Mor:survey}
\bysame, \emph{Structure of the mapping class groups of surfaces: a survey and
  a prospect}, Proceedings of the {K}irbyfest ({B}erkeley, {CA}, 1998), Geom.
  Topol. Monogr., vol.~2, Geom. Topol. Publ., Coventry, 1999, pp.~349--406.

\bibitem[MR97]{MR98}
G.~Masbaum and J.~D. Roberts, \emph{A simple proof of integrality of quantum
  invariants at prime roots of unity}, Math. Proc. Cambridge Philos. Soc.
  \textbf{121} (1997), no.~3, 443--454.

\bibitem[MR12]{MR12}
Gregor Masbaum and Alan~W. Reid, \emph{All finite groups are involved in the
  mapping class group}, Geom. Topol. \textbf{16} (2012), no.~3, 1393--1411.

\bibitem[MSS15]{MSS15}
Shigeyuki Morita, Takuya Sakasai, and Masaaki Suzuki, \emph{Structure of
  symplectic invariant {L}ie subalgebras of symplectic derivation {L}ie
  algebras}, Adv. Math. \textbf{282} (2015), 291--334.

\bibitem[MSS20]{MSS20}
\bysame, \emph{Torelli group, {J}ohnson kernel, and invariants of homology
  spheres}, Quantum Topol. \textbf{11} (2020), no.~2, 379--410.

\bibitem[Mur95]{Mu95}
Hitoshi Murakami, \emph{Quantum {${\rm SO}(3)$}-invariants dominate the {${\rm
  SU}(2)$}-invariant of {C}asson and {W}alker}, Math. Proc. Cambridge Philos.
  Soc. \textbf{117} (1995), no.~2, 237--249.

\bibitem[PS83]{PS83}
A.~A. Premet and I.~D. Suprunenko, \emph{The {W}eyl modules and the irreducible
  representations of the symplectic group with the fundamental highest
  weights}, Comm. Algebra \textbf{11} (1983), no.~12, 1309--1342.

\bibitem[Put07]{P07}
Andrew Putman, \emph{Cutting and pasting in the {T}orelli group}, Geom. Topol.
  \textbf{11} (2007), 829--865.

\bibitem[Rob01]{R01}
Justin Roberts, \emph{Irreducibility of some quantum representations of mapping
  class groups}, vol.~10, 2001, Knots in Hellas '98, Vol. 3 (Delphi),
  pp.~763--767.

\bibitem[RT91]{RT91}
N.~Reshetikhin and V.~G. Turaev, \emph{Invariants of {$3$}-manifolds via link
  polynomials and quantum groups}, Invent. Math. \textbf{103} (1991), no.~3,
  547--597.

\bibitem[Wit89]{W89}
Edward Witten, \emph{Quantum field theory and the {J}ones polynomial}, Comm.
  Math. Phys. \textbf{121} (1989), no.~3, 351--399.

\end{thebibliography}
\end{document}